\documentclass[a4paper,11pt]{article}
\usepackage{authblk}

\usepackage[english]{babel}
\usepackage[T1]{fontenc}
\usepackage{pifont}
\usepackage{color}

\usepackage{amsmath}
\usepackage{amssymb}
\usepackage{amsthm}
\usepackage{amscd}
\usepackage{mathrsfs}
\usepackage[all]{xy}
\xyoption{2cell}
\SelectTips{cm}{10}
\UseAllTwocells
\usepackage{enumerate}
\usepackage{makeidx}
\usepackage[applemac]{inputenc}
\usepackage{array}
\usepackage{multirow}
\usepackage{pgf,tikz,pgfplots}
\pgfplotsset{compat=1.15}
\usetikzlibrary{arrows}
\usepackage{relsize}

\usepackage{geometry}
\newcounter{c tilde enumeration}

\newtheorem{theorem}			     {Theorem}	    [section]
	
\newtheorem{corollary}	  [theorem]	 {Corollary}	
\newtheorem{lemma}	      [theorem]  {Lemma}		
\theoremstyle{definition}

\newtheorem*{remark*} 	  		     {Remark}

\newtheorem{example}	  [theorem]  {Example}	

\DeclareMathOperator{\Lex}{\mathbf{Lex}}
\DeclareMathOperator{\Set}{\mathbf{Set}}
\DeclareMathOperator{\Sk}{\mathbf{Sk}}
\DeclareMathOperator{\op}{\mathsf{op}}
\DeclareMathOperator{\ob}{\mathsf{ob}}
\DeclareMathOperator{\Graph}{\mathbf{FGraph}}

\DeclareMathOperator{\LC}{\Lex(\mathbb{C},\Set)^{\op}}
\DeclareMathOperator{\GG}{\mathscr{G}}
\DeclareMathOperator{\limit}{\mathsf{lim}}
\DeclareMathOperator{\colimit}{\mathsf{colim}}
\DeclareMathOperator{\Ob}{\mathsf{Ob}}
\DeclareMathOperator{\Pb}{\mathsf{Pb}}
\DeclareMathOperator{\Po}{\mathsf{Po}}
\DeclareMathOperator{\Prod}{\mathsf{Prod}}
\DeclareMathOperator{\Coprod}{\mathsf{Coprod}}
\DeclareMathOperator{\Eq}{\mathsf{Eq}}
\DeclareMathOperator{\Coeq}{\mathsf{Coeq}}
\renewcommand{\hom}{\mathsf{hom}}

\newcommand{\A}{\mathcal{A}}
\renewcommand{\AA}{\mathbb{A}}

\newcommand{\B}{\mathcal{B}}
\newcommand{\BB}{\mathbb{B}}

\newcommand{\CC}{\mathbb{C}}
\newcommand{\CCC}{\widetilde{\mathbb{C}}}

\newcommand{\DD}{\mathbb{D}}

\newcommand{\G}{\mathcal{G}}
\renewcommand{\H}{\mathcal{H}}

\newcommand{\NN}{\mathbb{N}}

\newcommand{\PP}{\mathbb{P}}

\newcommand{\V}{\mathcal{V}}

\newcommand{\X}{\mathcal{X}}
\newcommand{\XX}{\mathbb{X}}

\newcommand{\y}{\mathsf{y}}
\newcommand{\Z}{\mathcal{Z}}

\newcommand{\cd}{\xymatrix}
\renewcommand{\1}{\mathbf{1}}

\newcounter{steps}
\newcommand{\newstep}{Step~\refstepcounter{steps}\thesteps. }

\title{On stability of exactness properties under the pro-completion}

\author[a,b]{Pierre-Alain Jacqmin}
\author[c]{Zurab Janelidze}

\affil[a]{\small{\textit{Institut de Recherche en Math\'ematique et Physique, Universit\'e catholique de Louvain, Chemin du Cyclotron 2, B 1348 Louvain-la-Neuve, Belgium}}}
\affil[b]{\small{\textit{Department of Mathematics and Statistics, University of Ottawa, 150 Louis-Pasteur, K1N 6N5, Ottawa, Ontario, Canada}}}
\affil[c]{\small{\textit{Mathematics Division, Department of Mathematical Sciences, Stellenbosch University, Private Bag X1 Matieland 7602, South Africa}}}

\date{21 October 2020}

\begin{document}

\maketitle

\begin{abstract} 

In this paper we formulate and prove a general theorem of stability of exactness properties under the pro-completion, which unifies several such theorems in the literature and gives many more. The theorem depends on a formal approach to exactness properties proposed in this paper, which is based on the theory of sketches. Our stability theorem has applications in proving theorems that establish links between exactness properties, as well as in establishing embedding (representation) theorems for classes of categories defined by exactness properties.   
\end{abstract}



\section*{Introduction}

The construction  of pro-completion of a category is well known in mathematics. For instance, the pro-completion of the category of finite groups is the category of profinite groups, the pro-completion of the category of finite sets is the category of profinite spaces, and so on. We address the following question: which properties of the given category (and more generally, of an internal structure in the category) carry over to its pro-completion? 

If $\CC$ is a small finitely complete category, its pro-completion is the same as its free cofiltered limit completion, which is given by the restricted Yoneda embedding $\CC \hookrightarrow \LC$, where $\Lex(\CC,\Set)$ is the category of finite limit preserving functors from $\CC$ to $\Set$ (see~\cite{AGV,GU}). In the literature, many so-called `exactness properties' have been shown to be stable under this construction: if $\CC$ satisfies the given property, so does $\LC$. Among examples of such properties are the following (in each case, the cited reference is where the corresponding `stability' result was first established):
being regular~\cite{barr2}, coregular~\cite{DS}, additive~\cite{DS2}, abelian~\cite{DS2}, exact Mal'tsev with pushouts~\cite{BP}, coregular co-Mal'tsev~\cite{GP}, coextensive with pushouts~\cite{CPR}, and extensive~\cite{CPR}. We prove in this paper a general stability theorem, which includes all of the above examples and establishes stability of other fundamental exactness properties, such as being semi-abelian, regular Mal'tsev, coherent with finite coproducts, and many more. In some sense, our approach to proving the general stability theorem is analogous to the approach used in the particular cases mentioned above. The generality brings in heavy technicalities; these we have tackled using $2$-categorical calculus of natural transformations. As it can be expected, we use a generalization of the set-based case of a lemma from~\cite{DS} called the `uniformity lemma' (see also Lemma~5.1 in~\cite{makkai}); its detailed proof forms, in fact, a substantial part of the proof of our general stability theorem. And of course, we rely on classical results about pro-completion found in~\cite{AGV,GU}.

In order to formulate a general stability theorem, first we had to formalize the notion of an exactness property. Although the study of particular exactness properties is one of the main research directions in category theory, little has been done in terms of developing a general theory of exactness properties --- a theory that would be in similar relation to investigation of categories defined by particular exactness properties as, say, universal algebra is to investigation of various concrete algebraic structures. The recent work~\cite{GL} develops a unified approach to a certain type of exactness properties relevant mostly in logic and geometry. In~\cite{janelidze1, janelidze2, janelidze5}, first steps towards a unified approach to `algebraic' exactness properties were made (see also~\cite{janelidze6,jacqmin2}). The present work is a first step in studying exactness properties of both of these two types simultaneously, although our notion of an exactness property also has some limitations. Furthermore, we only take the theory as far as it is required for formulating and proving the stability theorem. A few topics for further investigation in the theory of exactness properties are suggested in the last section of the paper.

Our approach to formalizing the definition of an exactness property builds on the theory of sketches due to Ehresmann~\cite{ehresmann}. This is not surprising since, intuitively, an exactness property is a property of the behaviour of limits and colimits, whereas a sketch is the formal data of some limits and colimits. The key ingredient in our approach is the notion of an `exactness sequent'. It is a sequence of sketch inclusions
$$\cd{\X\ar[r]^-{\alpha} & \A\ar[r]^-{\beta} & \B}$$
which we abbreviate as $\alpha\vdash\beta$ to allude to its logical interpretation. Given a model $F$ of the sketch $\X$ in a category $\CC$, we define a `verification' of $\alpha\vdash\beta$ to be a map which assigns to each extension $G$ of $F$ along $\alpha$ an extension of $G$ along $\beta$. Most exactness properties of a category $\CC$ can be formalized as existence of verifications for sets of sequents, all of which start with the empty sketch $\X$. Thus, an exactness property of a category states that any $\A$-structure in the category admits a $\beta$-extension. For our theorem, we want verifications to be functorial, which is indeed the case in the main examples. When $\X$ is not the empty sketch, we get what can be seen as an exactness property of an internal structure in a category. This includes examples such as an internal monoid being an internal group, a morphism being the truth morphism for a subobject classifier, a split extension being a split-extension classifier in the sense of~\cite{BJK}, and others. Our approach to exactness properties does not cover all properties of a category that is of interest. It can rather be thought of as formalization of the so-called `first-order' exactness properties. An example of a `higher-order' exactness property would be the property of existence of enough projectives, whereas for an object $P$ to be a projective object would be a first-order exactness property of $P$ (see the last section of the paper for further remarks about the order of exactness properties). According to our stability theorem, not all but only certain first-order exactness properties are stable under the pro-completion. A counterexample is given by the exactness property of a morphism to be the truth morphism for a subobject classifier. Another counterexample is for a category to be exact in the sense of~\cite{barr}. We thus show that, under the conditions of our stability theorem, given a functorial verification of $\alpha\vdash\beta$ for an $\X$-structure $F$ in~$\CC$, there exists a functorial verification of $\alpha\vdash\beta$ for the image $\y F$ of $F$ under the (restricted) Yoneda embedding $\y\colon \CC \hookrightarrow \LC$. Moreover, the functorial verification of $\alpha\vdash\beta$ for $\y F$ can be chosen so that it is `coherent' with (agrees with) the functorial verification of $\alpha\vdash\beta$ for~$F$.

As far as applications of the stability theorem are concerned, we have the following:
\begin{itemize}
\item The stability theorem allows to apply categorical proofs involving colimits to categories which do not necessarily have colimits. 
This has been explained and used in~\cite{jacqminthesis,JR}. Roughly speaking, it goes as follows. Consider two exactness properties $P$ and $Q$ expressed in terms of finite limits. Suppose one has a proof that the implication $P\Rightarrow Q$ holds for any finitely complete and cocomplete category. An obvious question then arises: does this implication hold for any finitely complete category? If one can prove that the exactness property $P$ is stable under the pro-completion, we can proceed as follows: let $\CC$ be a finitely complete category satisfying $P$. By the axiom of universes, one can assume it is small. Then, its free cofiltered limit completion $\LC$ satisfies $P$. Since $\LC$ is complete and cocomplete, it also satisfies $Q$. And since the Yoneda embedding $\CC \hookrightarrow \LC$ preserves finite limits and all colimits, and reflects isomorphisms, one can usually show that $\CC$ also satisfies $Q$. It is worth mentioning that the example given in~\cite{JR} is quite involved and no direct proof of it has been found for now.

\item The stability theorem opens a way to new embedding theorems in categorical algebra. Barr proved and used in~\cite{barr2} a particular instance of the stability theorem for the property of being a regular category. This was a crucial step in proving his embedding theorem for regular categories. In a similar way, while this paper was under preparation, other particular instances of our stability theorem, together with the theory of `approximate operations' originating in~\cite{boujan mal,janelidze6}, enabled the first author to establish embedding theorems for other classes of categories such as regular Mal'tsev categories  in~\cite{jacqminthesis,jacqmin1,jacqmin2}. These theorems often provide a better technique for proving theorems in general categories than the one described above.
\end{itemize}

Finally, let us remark that applying our stability theorem to a particular exactness property is not always a straightforward task. The obvious presentation of the exactness property in terms of a set of sequents may not give sequents that fulfil the requirements in our theorem. Nevertheless, sometimes it becomes possible to appropriately reformulate the exactness property. When even that is not achievable, it may still be possible to slightly strengthen the property with other exactness properties and then give it a representation as a set of sequents admissible for the theorem. For instance, we do not know if our theorem can be applied to (finitely complete) Mal'tsev categories~\cite{CPP}, while it is applicable to regular Mal'tsev categories. This and some other examples of this nature are detailed at the end of the first section of the paper.

\begin{remark*}
Earlier unpublished draft versions of this paper have been cited as `Unconditional exactness properties' in~\cite{jacqminthesis,jacqmin1} and as `Functorial exactness properties' in~\cite{JR,jacqmin2}.
\end{remark*}

\subsection*{Acknowledgement}

The first author would like to thank Stellenbosch University for its kind hospitality during his first visit in 2014 when the present project started, and during his second visit in 2020. He also thanks the Belgian FNRS and the Canadian NSERC for their generous support. The second author is grateful to the South African NRF for its financial support. He is also grateful to the kind hospitality of University of Louvain-la-Neuve during his several visits when part of the collaboration on this paper took place.

\tableofcontents

\section{A formal approach to exactness properties}\label{section main theorem}

\subsection{Commutativity and convergence conditions}

Let $\G$ be a graph, i.e., a diagram $d,c\colon E\rightrightarrows V$ in $\Set$, the category of sets. By a \emph{path} in $\G$, we mean, as usual, an alternating sequence $(A_0,f_1,A_1,\dots,f_n,A_n)$ of vertices and arrows with $n\geqslant 0$, $d(f_i)=A_{i-1}$ and $c(f_i)=A_i$ for each $i\in\{1,\dots,n\}$. As in~\cite{borceux}, a \emph{commutativity condition} in $\G$ is a pair of paths $$((A_0,f_1,A_1,\dots,f_n,A_n),(B_0,g_1,B_1,\dots,g_m,B_m))$$ in $\G$ such that $A_0=B_0$ and $A_n=B_m$. We will represent it by $$f_n \cdots f_1 = g_m \cdots g_1$$ or by $$f_n\cdots f_1=1_{B_0}$$ if $m=0$ (and similarly if $n=0$).
A \emph{finite diagram} in $\G$ is given by a finite graph $\H$ together with a morphism of graphs $D\colon \H \to \G$. A \emph{finite limit condition} (respectively, a \emph{finite colimit condition}) in $\G$ is an equivalence class of 4-tuples $(\H,D,C,(c_H)_{H\in\H})$ where $D\colon \H \to \G$ is a finite diagram, $C$ is an object in $\G$ and for each object $H$ in $\H$, $c_H\colon C \to D(H)$ (respectively, $c_H\colon D(H)\to C$) is an arrow in $\G$. Two such 4-tuples $(\H,D,C,(c_H)_{H\in\H})$ and $(\H',D',C',(c'_{H'})_{H'\in\H'})$ are considered to be equivalent if $C=C'$ and if there exists an isomorphism of graphs $I\colon \H\to\H'$ such that $D'I=D$ and $c_H=c'_{I(H)}$ for any $H\in\H$. Such a condition $[(\H,D,C,(c_H)_{H\in\H})]$ will be represented by $$(C,(c_H)_H) = \limit (\H,D) \quad \text{(respectively, by } (C,(c_H)_H) = \colimit (\H,D) \text{).}$$
Finite limit conditions and finite colimit conditions are called \emph{convergence conditions}.

\subsection{Sketches}

The theory of sketches is due to Ehresmann~\cite{ehresmann}. Our approach to sketches differs slightly from his, but only at the level of presentation. We define an \emph{exactness sketch} (or simply a \emph{sketch}) as a finite graph equipped with a set of commutativity conditions and a set of convergence conditions. A \emph{morphism of sketches} is a morphism $\mu\colon \G \to \G'$ of underlying graphs of sketches which carries each commutativity condition on $\G$ to a commutativity condition on $\G'$ and each convergence condition on $\G$ to a convergence condition on $\G'$. With the obvious way of composing morphisms of sketches, we obtain the category $\Sk$ of sketches. The forgetful functor $$\GG\colon \Sk\to\Graph,$$ which maps each sketch to its underlying graph in the category $\Graph$ of finite graphs, is a `topological functor' (see e.g.~\cite{B84}), that is, both $\GG$ and $\GG^{\op}$ are fibrations~\cite{Gro59} whose fibres are complete lattices.

A \emph{subsketch} of a sketch $\B$ is a subgraph $\A$ of the underlying graph of $\B$, equipped with a sketch structure that turns the inclusion of graphs $\A\to\B$ into a sketch morphism. We will call such morphisms \emph{subsketch inclusions}. By a \emph{regular subsketch} of a sketch we mean a subsketch for which the corresponding subsketch inclusion $\beta\colon\A\to\B$ is a regular monomorphism in $\Sk$. It is easy to see that this is equivalent to $\beta$ being a cartesian morphism for the functor $\GG$ (in simpler terms, $\A$ inherits all conditions of $\B$ that can be expressed in $\A$).

\subsection{Exactness structures}\label{secA}

Given a sketch $\X$ and a category $\CC$, an \emph{exactness structure of type $\X$} (or simply, an \emph{$\X$-structure}) in $\CC$ is a morphism $F\colon \X\to\CC$ of graphs which carries each commutativity condition of $\X$ to an actual commutative diagram in $\CC$, and each finite limit/colimit condition of $\X$ to an actual limit/colimit in $\CC$. In a given category $\CC$, structures of the same type $\X$ and natural transformations between them form a category, under the usual composition of natural transformations. We denote this category by $\X\CC$. Every morphism $\varphi\colon \A\to\B$ of sketches gives rise to a functor $$\varphi_\CC\colon\B\CC\to\A\CC$$ of `composition with $\varphi$', defined by the mapping $F\mapsto F\circ\varphi$ for structures and a similar one for their transformations. For an $\A$-structure $G$ in $\CC$, we write $\B^\varphi_G\CC$ to denote the fibre of $\varphi_\CC$ at $G$.

\subsection{Exactness sequents}\label{secB}

The notion of an `exactness sequent' introduced here is new and it allows a formal approach to `(finitary) exactness properties'.
For a sketch $\X$, an \emph{exactness sequent of type $\X$} (or simply, an \emph{$\X$-sequent}) is a sequence
\begin{equation}\label{EquP}
\cd{\X\ar[r]^-{\alpha} & \A\ar[r]^-{\beta} & \B}
\end{equation}
of subsketch inclusions, abbreviated as
$$\alpha\vdash\beta.$$
Let $F$ be an $\X$-structure in a category $\mathbb{C}$. 
A \emph{verification} of an $\X$-sequent $\alpha\vdash\beta$ for $F$ is a right inverse of the object function of the restriction
$$\beta_F^{\alpha}\colon \B^{\beta\alpha}_F\CC\to \A^\alpha_F\CC$$
of the functor $\beta_\CC\colon\B\CC\to\A\CC$. An actual right inverse functor of the same functor is called a \emph{functorial verification} of $\alpha\vdash\beta$ for $F$.

We will sometimes write `$\alpha\vdash_F \beta$' as an abbreviation of the statement `there exists a verification of $\alpha\vdash \beta$ for $F$' and `$\alpha\vdash_F \beta$ functorially' for `there exists a functorial verification of $\alpha\vdash \beta$ for $F$'.

If $\X=\varnothing$ is the empty sketch and $\CC$ a category, there is a unique $\varnothing$-structure in $\CC$. A (functorial) verification of an $\varnothing$-sequent $\alpha\vdash\beta$ for this unique $\varnothing$-structure $F$ will be simply called a (functorial) verification of $\alpha\vdash\beta$ for $\CC$. We write in this case $\alpha\vdash_\CC\beta$ instead of $\alpha\vdash_F\beta$.

Notice that in general, an exactness sequent may have several (functorial) verifications for the same $\X$-structure. We will describe in Subsection~\ref{section constructible sequents} a particular case of exactness sequents for which this cannot happen, and moreover, for which verifications are always extendable to functorial verifications. In this case, existence of a (functorial) verification becomes a property. Classical exactness properties fall under this case with further $\X=\varnothing$.  Note that by allowing non-empty $\X$, we are generalizing exactness properties to internal structures in a category. See Subsection~\ref{section examples} for examples.

\subsection{Equivalent conditions for functorial verification}

Let us now prove two easy lemmas that will be used later.

\begin{lemma}\label{LemB}
Let $\beta\colon\A\to\B$ be a morphism of sketches which is injective on objects. Given a category $\CC$ and an isomorphism $i\colon \beta_\CC(H) \to G$ in $\A\CC$, there exists a $\B$-structure $E\in\B^\beta_G\CC$ and an isomorphism $j\colon H \to E$ in $\B\CC$ such that $\beta_\CC(j)=i$.
\end{lemma}

\begin{proof}
We denote by $\beta(\ob\A)$ the image in the set of objects of $\B$ of the object function part of $\beta$. Since $\beta$ is injective on objects, we can define $E$ on objects as
$$E(B)=\begin{cases} G(A), & \textrm{if }B=\beta(A), \\ H(B), & \textrm{if }B\notin\beta(\ob\A).\end{cases}$$
We can also define for an object $B$ in $\B$,
$$j_B=\begin{cases} i_A, & \textrm{if }B=\beta(A), \\ 1_{H(B)}, & \textrm{if }B\notin\beta(\ob\A).\end{cases}$$
For an arrow $b\colon B_1\to B_2$ in $\B$, we define $E(b)$ as $$E(b)=j_{B_2}H(b)j^{-1}_{B_1}.$$
Then $E$ is a graph morphism $E\colon \B\to\CC$ and $j_B$ is an isomorphism natural in $B$. Since $H$ is a $\B$-structure and $E$ is naturally isomorphic to it (via $j$), also $E$ is a $\B$-structure. It is also clear that $\beta_\CC(j)=i$.
By the definitions of $E$ and $j$, and by naturality of $i$, we get that for any arrow $a\colon A_1\to A_2$ in~$\A$, the following diagram commutes:
$$\xymatrix@=50pt{ E(\beta(A_1)) \ar[rrr]^{E(\beta(a))} \ar[dr]^-{j^{-1}_{\beta(A_1)}} \ar@{=}[ddr] & & & E(\beta(A_2)) \ar@{=}[ddl] \ar[dl]_-{j^{-1}_{\beta(A_2)}} \\ & H(\beta(A_1)) \ar[d]^-{i_{A_1}} \ar[r]^-{H(\beta(a))} \ar@{}[ur]|-{=} & H(\beta(A_2)) \ar[d]_-{i_{A_2}} & \\ & G(A_1) \ar[r]_-{G(a)} \ar@{}[ur]|-{=} & G(A_2) &}$$
This shows $E(\beta(a))=G(a)$. Thus, $E\in\B^\beta_G\CC$ and the proof is complete.
\end{proof}

\begin{lemma}\label{LemA}
An $\X$-sequent $\alpha\vdash\beta$ admits a functorial verification for an $\X$-structure $F$ in a category $\CC$ if and only if for each $G\in\A^\alpha_F\CC$ there exists $H_G\in\B\CC$, which depends on $G$ functorially (over the category $\A^\alpha_F\CC$) and for which there is an isomorphism $H_G\circ \beta\cong G$ natural in~$G$. 
\end{lemma}

\begin{proof}
The `only if' part is obvious. For the `if part', let $i^G$ denote the isomorphism $H_G\circ\beta \to G$ for each $G\in\A^\alpha_F\CC$. Let also $j^G\colon H_G \to E(G)$ be the isomorphism given by Lemma~\ref{LemB}. Finally, for a morphism $m\colon G_1\to G_2$ in $\A^\alpha_F\CC$, we set $$E(m)=j^{G_2}H_m(j^{G_1})^{-1}.$$
This defines a functor $$E\colon \A^\alpha_F\CC\to \B^{\beta\alpha}_F\CC$$
which is a right inverse of the functor $\beta_F^{\alpha}$ since $\beta_\CC(j^G)=i^G$ for each $G\in\A^\alpha_F\CC$ and $i^G$ is natural in~$G$:
$$\beta_F^{\alpha}(E(m))=\beta_\CC(j^{G_2})\beta_\CC(H_m)\beta_\CC(j^{G_1})^{-1}=i^{G_2}\beta_\CC(H_m)(i^{G_1})^{-1}=m$$
for each morphism $m\colon G_1\to G_2$ in $\A^\alpha_F\CC$.
\end{proof}

\subsection{Unconditionality} 

To each finite category $\AA$ we associate a sketch $\A$, called the \emph{underlying sketch} of $\AA$. The underlying graph of $\A$ is the same as the underlying graph of $\AA$ and the commutativity conditions are given by
\begin{itemize}
\item $((A,f,B,g,C),(A,gf,C))$ for any pair of composable arrows $f\colon A \to B$ and $g\colon B\to C$ in $\AA$;
\item $((A,1_A,A),(A))$ for any object $A$ in $\AA$.
\end{itemize}
There are no convergence conditions on $\A$.

A subsketch inclusion $\alpha\colon \X\to\A$ is said to be \emph{unconditional of finite kind} if $\A$ is the underlying sketch of a finite category further equipped with the commutativity and convergence conditions which already appear in $\X$ (and no others). The stability theorem will concern exactness sequents $\alpha\vdash\beta$ where $\alpha$ is unconditional of finite kind.

\subsection{Constructibility}\label{section constructible sequents}

Before giving some concrete examples of exactness sequents, let us consider a particular case for which functoriality of the right inverse of $\beta^\alpha_F$ is a consequence of the fact that $\beta^\alpha_F$ is essentially surjective.

A subsketch inclusion $\beta\colon\A\to\B$ is said to be \emph{constructible} if it is the composite of a finite sequence
$$\cd{\A\ar[r] & \bullet\ar[r] & \bullet\ar[r] & \cdots\ar[r] &\B  }$$
of subsketch inclusions, where every next subsketch of $\B$ is obtained from the previous one by any one of the following procedures:
\begin{itemize}
\item include some commutativity and convergence conditions from $\B$ expressed using objects and arrows which belong to the subsketch; 

\item include an arrow $f$ from $\B$ and a commutativity condition $f=g_n\cdots g_1$ from $\B$, for $n\geqslant 0$ and existing arrows $g_1,\dots,g_n$ in the subsketch;

\item include an object $C$ from $\B$, not already in the subsketch, together with the arrows $c_H$ and the condition $(C,(c_H)_H)=\limit (\H,D)$ from $\B$, where $D$ is a finite diagram in the subsketch;

\item include an object $C$ from $\B$, not already in the subsketch, together with the arrows $c_H$ and the condition $(C,(c_H)_H)=\colimit (\H,D)$ from $\B$, where $D$ is a finite diagram in the subsketch;

\item given in the subsketch a condition $(C,(c_H)_H)=\limit (\H,D)$, an object $X$, a family $(x_H\colon X\to D(H))_{H\in\H}$ of arrows and commutativity conditions $D(h)\cdot x_H =  x_{H'}$ for each arrow $h\colon H\to H'$ in $\H$, include from $\B$ an arrow $f\colon X\to C$ and commutativity conditions $c_H \cdot f = x_H$ for each object $H\in\H$;

\item given in the subsketch a condition $(C,(c_H)_H)=\colimit (\H,D)$, an object $X$, a family $(x_H\colon D(H)\to X)_{H\in\H}$ of arrows and commutativity conditions $x_{H'} \cdot D(h) = x_H$ for each arrow $h\colon H\to H'$ in $\H$, include from $\B$ an arrow $f\colon C\to X$ and commutativity conditions $f \cdot c_H = x_H$ for each object $H\in\H$.
\end{itemize}

\begin{lemma}\label{lemma on unconditional exactness properties}
For any constructible subsketch inclusion $\beta\colon\A\to\B$ and any category $\CC$, the functor $\beta_\CC$ is full and faithful. Moreover, for any $\A$-structure $G$ in $\CC$, if there exists a $\B$-structure $H$ in $\CC$ such that $\beta_\CC(H)$ is isomorphic to $G$, then the fibre $\B^\beta_G\CC$ is non-empty.
\end{lemma}

\begin{proof}
In view of the above definition, this reduces to easy verification that the stated properties hold for any subsketch inclusion $\beta\colon\A\to\B$, where $\A$ is a subsketch of $\B$ such that applying one of the above procedures results in $\B$.
\end{proof}

In particular, if
$$\cd{\X\ar[r]^-{\alpha} & \A\ar[r]^-{\beta} & \B}$$
is an exactness sequent with $\beta$ a constructible subsketch inclusion and if $F$ is an $\X$-structure in a category $\CC$, the functor $\beta_{\CC}\colon\B\CC\to\A\CC $ restricts to a full and faithful functor
$$\beta^\alpha_F\colon \B^{\beta\alpha}_F\CC \to \A^\alpha_F\CC.$$
Moreover, $F$ admits a functorial verification of $\alpha\vdash\beta$ if and only if $\beta^\alpha_F$ is essentially surjective on objects (and hence an equivalence of categories), proving the following lemma.

\begin{lemma}\label{lemma verification constructible}
Let $\cd{\X\ar[r]^-{\alpha} & \A\ar[r]^-{\beta} & \B}$ be an exactness sequent such that $\beta$ is a constructible subsketch inclusion. For an $\X$-structure $F$ in a category $\CC$, the following statements are equivalent
\begin{enumerate}[(i)]
\item $F$ admits a functorial verification of $\alpha\vdash\beta$, i.e., $\alpha\vdash_F\beta$ functorially;
\item $F$ admits a verification of $\alpha\vdash\beta$, i.e., $\alpha\vdash_F\beta$;
\item for all $\A$-structure $G$ in $\CC$ such that $G \circ \alpha = F$, there exists a $\B$-structure $H$ in $\CC$ such that $H \circ \beta$ is isomorphic to $G$.
\end{enumerate}
Moreover, if those conditions hold, there are, up to isomorphisms, exactly one verification and exactly one functorial verification of $\alpha\vdash\beta$ for $F$.
\end{lemma}

\subsection{Dual sequents}\label{section dual sequents}

Each sketch $\Z$ gives rise to a dual sketch $\Z^{\op}$: the underlying graph of $\Z^{\op}$ is the dual of the underlying graph of $\Z$, each commutativity condition $f_n \cdots f_1 = g_m \cdots g_1$ in $\Z$ is turned into a condition $f_1^{\op} \cdots f_n^{\op}=g_1^{\op} \cdots g_m^{\op}$ in $\Z^{\op}$, each finite limit condition $(C,(c_H)_H)=\limit (\H,D)$ in $\Z$ is turned into a finite colimit condition $(C,(c_H^{\op})_H)=\colimit (\H^{\op},D^{\op})$ in $\Z^{\op}$ and vice-versa. As usual, a morphism of sketches $\mu\colon \V\to\Z$ gives rise to a morphism $\mu^{\op}\colon\V^{\op}\to\Z^{\op}$ between the dual sketches. Similarly, each $\Z$-structure $G$ in a category $\CC$ can be turned into a $\Z^{\op}$-structure $G^{\op}$ in~$\CC^{\op}$. This gives an isomorphism of categories $\Z^{\op}(\CC^{\op})\cong (\Z\CC)^{\op}$. It is then not hard to see that given an exactness sequent $\alpha\vdash\beta$ as in~(\ref{EquP}) and an $\X$-structure $F$ in $\CC$, we have $\alpha\vdash_F\beta$ if and only if $\alpha^{\op}\vdash_{F^{\op}}\beta^{\op}$, and $\alpha\vdash_F\beta$ functorially if and only if $\alpha^{\op}\vdash_{F^{\op}}\beta^{\op}$ functorially. Notice also that a subsketch inclusion $\alpha$ is unconditional of finite kind if and only if $\alpha^{\op}$ is. Moreover, a subsketch inclusion $\beta$ is regular (respectively, constructible) if and only if $\beta^{\op}$ is.

\subsection{Concrete examples}\label{section examples}

The selection of examples of exactness properties included here are for illustration only and by no means do we provide a comprehensive list of examples. Investigation of exactness properties is one of the central activities in research in category theory. New exactness properties have been arising in the literature since the birth of the subject of category theory in~\cite{eilmac}. The first exactness properties expressed properties of the modern-day notion of an abelian category, and go back to~\cite{maclane}. Our list of examples contains a selection from classical exactness properties to ones arising in recent literature. We do not claim any priority of these examples over others that have not been mentioned in this paper.

In this section we will show how certain exactness properties can be concretely represented by exactness sequents $\alpha\vdash\beta$. Unless stated otherwise, in all of these sequents, $\alpha$ will be unconditional of finite kind and $\beta$ will be constructible.

\begin{example}\label{example isomorphism} In this example, we describe an exactness sequent that encodes the property of a morphism to be an isomorphism. If $\X$ is the sketch $$\cd{A \ar[r]^-{f} & B}$$ with no conditions, an $\X$-structure in a category $\CC$ is just a morphism in that category. Let $\A$ be the underlying sketch of the arrow category: $$\cd{A \ar[r]^-{f} \ar@(ul,dl)_-{1_A} & B \ar@(ur,dr)^-{1_B}}$$ Let $\B$ be constructed by adding to $\A$ the finite limit condition $(A,(f))=\limit (\H_{\Ob},D_{\Ob}^B)$ where $\H_{\Ob}$ is the graph with $W$ as unique object and without any arrows and $D_{\Ob}^B \colon \H_{\Ob} \to \GG(\A)$ is defined by $D_{\Ob}^B(W)=B$. Then, an $\X$-structure $F$ in the category $\CC$ (i.e., a morphism $F(f)$ in $\CC$) admits a (functorial) verification of $\alpha\vdash\beta$ exactly when the cone $\cd{F(A) \ar[r]^-{F(f)} & F(B)}$ is a limit over the single object diagram $F(B)$, that is, when $F(f)$ is an isomorphism.
\end{example}

\begin{example} In this example, we describe an exactness sequent that encodes the property of a morphism to be a monomorphism.
The sketches $\X$ and $\A$ are here as in Example~\ref{example isomorphism}. Now, the sketch $\B$ is obtained by adding to $\A$ the convergence condition $(A,(1_A,1_A,f)) = \limit (\H_{\Pb},D_{\Pb}^{f,f})$
where $\H_{\Pb}$ is the graph $$\cd{&W_2 \ar[d]^-{w_2} \\ W_1 \ar[r]_-{w_1} & W_3}$$
and $D_{\Pb}^{f,f} \colon \H_{\Pb} \to \GG(\A)$ is defined via $D_{\Pb}^{f,f}(w_1)=D_{\Pb}^{f,f}(w_2)=f$. In this case, an $\X$-structure $F$ in~$\CC$ (i.e., a morphism $F(f)$ in $\CC$) admits a (functorial) verification of $\alpha\vdash\beta$ if and only if the square $$\cd{F(A) \ar[r]^-{1_{F(A)}} \ar[d]_-{1_{F(A)}} & F(A) \ar[d]^-{F(f)} \\ F(A) \ar[r]_-{F(f)} & F(B)}$$
is a pullback, that is, when the morphism $F(f)$ is a monomorphism.
\end{example}

In the forthcoming examples, we will specify sketches and categories by incomplete drawings according to the following rules:
\begin{itemize}
\item We may omit the identity arrows in the drawings of categories and in the drawings of graphs that contain underlying graphs of categories as subgraphs. 

\item We may omit the composite morphisms of morphisms which are already displayed in the drawings of categories and in the drawings of graphs that contain underlying graphs of categories as subgraphs.

\item We may display the convergence conditions of sketches by listing the (co)limits they represent. In particular, in a sketch $\Z$ with underlying graph $\G$, we will use the usual abbreviations~\ref{abb iso}--\ref{abb cokernel pair} set out below.

\item For the abbreviations~\ref{abb equalizer}--\ref{abb cokernel pair}, the arrow $h$ will often be omitted in the display of~$\G$.
\end{itemize}
The abbreviations for displaying convergence conditions in a sketch are:
\begin{enumerate}[1.]
\item\label{abb iso} `$f\colon A\to B$ represents an isomorphism' means the condition $(A,(f))=\limit (\H_{\Ob},D_{\Ob}^B)$ where $\H_{\Ob}$ is the graph with one object $W$ and no arrows and $D_{\Ob}^B\colon\H_{\Ob}\to\G$ is defined by $D_{\Ob}^B(W)=B$.   
\item\label{abb terminal} `$A$ represents the terminal object' means the condition $(A,\varnothing)=\limit (\varnothing,D_{!})$ where $D_{!}\colon \varnothing\to\G$ is the unique graph morphism from the empty graph $\varnothing$ to~$\G$.
\item\label{abb initial} `$A$ represents the initial object' means the condition $(A,\varnothing)=\colimit (\varnothing,D_{!})$ where $D_{!}\colon \varnothing\to\G$ is as above.
\item\label{abb product} `$(P,p_1,p_2)$ represents the product of $A$ and $B$' means the condition $(P,(p_1,p_2))=\limit (\H_{\Prod},D_{\Prod}^{A,B})$ where $\H_{\Prod}$ is the graph with two objects $W_1$ and $W_2$ and no arrows, and $D_{\Prod}^{A,B}\colon\H_{\Prod}\to\G$ is defined by $D_{\Prod}^{A,B}(W_1)=A$ and $D_{\Prod}^{A,B}(W_2)=B$.
\item\label{abb coproduct} `$(C,i_1,i_2)$ represents the coproduct of $A$ and $B$' means the condition $(C,(i_1,i_2))=\colimit (\H_{\Coprod},D_{\Coprod}^{A,B})$ where $\H_{\Coprod}=\H_{\Prod}$ and $D_{\Coprod}^{A,B}=D_{\Prod}^{A,B}$ are as above.
\item\label{abb equalizer} `$(E,e)$ represents the equalizer of $f$ and $g$' means, for some suitable arrow $h$, the condition $(E,(e,h))=\limit (\H_{\Eq},D_{\Eq}^{f,g})$ where $\H_{\Eq}$ is the graph $$\cd{W_1 \ar@<4pt>[r]^-{w_1} \ar@<-3pt>[r]_-{w_2} & W_2}$$
and $D_{\Eq}^{f,g}\colon\H_{\Eq}\to\G$ is defined via $D_{\Eq}^{f,g}(w_1)=f$ and $D_{\Eq}^{f,g}(w_2)=g$.
\item\label{abb coequalizer} `$(Q,q)$ represents the coequalizer of $f$ and $g$' means, for some suitable arrow $h$, the condition $(Q,(h,q))=\colimit (\H_{\Coeq},D_{\Coeq}^{f,g})$ where $\H_{\Coeq}=\H_{\Eq}$ and $D_{\Coeq}^{f,g}=D_{\Eq}^{f,g}$ are as above.
\item\label{abb pullback} `$(P,p_1,p_2)$ represents the pullback of $f$ along $g$' means, for some suitable arrow $h$, the condition $(P,(p_1,p_2,h))=\limit (\H_{\Pb},D_{\Pb}^{f,g})$ where $\H_{\Pb}$ is the graph $$\cd{&W_2 \ar[d]^-{w_2} \\ W_1 \ar[r]_-{w_1} & W_3}$$ and $D_{\Pb}^{f,g} \colon \H_{\Pb} \to \G$ is defined via $D_{\Pb}^{f,g}(w_1)=f$ and $D_{\Pb}^{f,g}(w_2)=g$.
\item\label{abb pushout} `$(Q,q_1,q_2)$ represents the pushout of $f$ along $g$' means, for some suitable arrow $h$, the condition $(Q,(q_1,q_2,h))=\colimit (\H_{\Po},D_{\Po}^{f,g})$ where $\H_{\Po}$ is the graph $$\cd{W_3 \ar[r]^-{w_2} \ar[d]_-{w_1} &W_2 \\ W_1 &}$$ and $D_{\Po}^{f,g} \colon \H_{\Po} \to \G$ is defined via $D_{\Po}^{f,g}(w_1)=f$ and $D_{\Po}^{f,g}(w_2)=g$.
\item\label{abb kernel pair} `$(R,r_1,r_2)$ represents the kernel pair of $f$' means `$(R,r_1,r_2)$ represents the pullback of $f$ along $f$'.
\item\label{abb cokernel pair} `$(Q,q_1,q_2)$ represents the cokernel pair of $f$' means `$(Q,q_1,q_2)$ represents the pushout of $f$ along $f$'.
\end{enumerate}

\begin{example}\label{existence of a binary product} In this example, we describe an exactness sequent that encodes the property of two objects to have a product.
Let $\X=\A$ be the underlying sketch of the category with two objects $X$ and $Y$ and no non-identity arrows. Let $\beta$ be the inclusion of $\A$ in the sketch
$$\begin{array}{cl} 
\vcenter{\cd{& P \ar[ld]_-{p_1} \ar[rd]^-{p_2} & \\ X && Y}} & \begin{array}{l} \text{conditions from }\A\text{ together with:} \\ (P,p_1,p_2) \text{ represents the product of }X \text{ and }Y.\end{array}\end{array}$$
An $\X$-structure is just the data of two objects and it admits a (functorial) verification of $\alpha\vdash\beta$ if and only if their product exists.
\end{example}

\begin{example} In this example, we describe an exactness sequent that encodes the property of a category to have all binary products. 
Let $\X$ be the empty sketch and let $\A$ and $\B$ be as in the Example~\ref{existence of a binary product}. Then, a category $\CC$ admits a (functorial) verification of $\alpha\vdash\beta$ exactly when it has all binary products.
\end{example}

The aim of the next example is to warn the reader that, in a convergence condition $(C,(c_H)_H) = \limit (\H,D)$, the arrows $c_H$ are not required to be pairwise distinct.

\begin{example} In this example, we describe an exactness sequent that encodes the property of an object to have at most one morphism from each object to it.
Let $\X=\A$ be the underlying sketch of the category with a single object $X$ and no non-identity arrows. Let $\beta$ be the inclusion of $\A$ in the sketch
$$\begin{array}{cl}
\vcenter{\cd{P \ar[r]^-{p} & X}} & \begin{array}{l} \text{conditions from }\A\text{ together with:} \\ (P,p,p) \text{ represents the product of }X \text{ and }X.\end{array}\end{array}$$
Then, for an $\X$-structure $F$ in a category $\CC$ (i.e., an object $F(X)$ of $\CC$), we have $1_\X\vdash_F\beta$ (functorially) if and only if, for each object $A$ in $\CC$, there is at most one morphism $A \to F(X)$.
\end{example}

\begin{example}\label{example reflexivity} In this example, we describe two exactness sequents both of which encode the property of a binary relation to be reflexive.
Let $\X=\A$ be the underlying sketch of the category
$$\cd{R \ar@<4pt>[r]^-{r_1} \ar@<-3pt>[r]_-{r_2} & X}$$
equipped with the convergence condition attesting that $(R,r_1,1_R,r_2,1_R)$ represents the limit of the outer square in the following diagram.
$$\cd{X && R \ar[ll]_-{r_1} \ar[dd]^-{r_2} \\ & R \ar[lu]_-{r_1} \ar[ru]_-{1_R} \ar[rd]_-{r_2} \ar[ld]_-{1_R}& \\ R \ar[uu]^-{r_1} \ar[rr]_-{r_2} && X}$$
An $\X$-structure $F$ in a category $\CC$ is a relation $F(R) \colon F(X) \nrightarrow F(X)$. We can express the condition that this relation is reflexive as the condition $1_{\X}\vdash_F\beta$ (functorially) for an exactness sequent $1_{\X}\vdash\beta$ where $\beta$ is the inclusion of $\A=\X$ in the sketch $\B$ given by
$$\begin{array}{cl}
\vcenter{\cd{R \ar@<5pt>[r]^-{r_1} \ar@<-2pt>[r]_-{r_2} & X \ar@/^1.1pc/@<2pt>[l]^-{e}}} & \begin{array}{l}
\text{conditions from }\A\text{ together with:} \\
r_1 \cdot e = 1_X,\\
r_2 \cdot e = 1_X.\end{array}\end{array}$$
Note that, in contrast with the above examples, $\beta$ is not here a constructible subsketch inclusion.
However, the existence of a verification is still equivalent to the existence of a functorial verification and such verifications are uniquely determined. Nonetheless, there is a way of presenting the reflexivity of a relation as the condition $1_{\X}\vdash_F\beta$ (functorially) for an exactness sequent $1_{\X}\vdash\beta$ where $\beta$ is constructible: Let again $\A=\X$ be as above but let now $\B$ be the sketch
$$\begin{array}{cl}
\vcenter{\cd{L \ar[d]_-{l_1} \ar@<1pt>[rd]^-{l_2} & \\ R \ar@<3pt>[r]^-{r_1} \ar@<-4pt>[r]_-{r_2} & X}} &
\begin{array}{l}
\text{conditions from }\A\text{ together with:} \\
(L,l_1,l_2) \text{ represents the equalizer of }r_1 \text{ and }r_2,\\
l_2\colon L\to X\text{ represents an isomorphism}.\end{array}\end{array}$$
A relation $F(R)$ admits a (functorial) verification of $\alpha\vdash\beta$ when the equalizer of $F(r_1)$ and $F(r_2)$ exists and is such that $F(l_2)$ is an isomorphism. This happens exactly when $F(R)$ is a reflexive relation.
\end{example}

From now on we will treat only those examples where $\X=\varnothing$ is the empty sketch. We recall that a (functorial) verification of an exactness sequent of type $\varnothing$ for a category $\CC$ is just a (functorial) verification for the unique $\varnothing$-structure in $\CC$. Since in all these examples, unless stated otherwise, $\alpha$ is unconditional of finite kind, instead of describing the sketch~$\A$, we will describe the (unique) finite category $\AA$ whose underlying sketch is~$\A$.

\begin{example} We show that the property of a category being regular~\cite{barr} is equivalent to the property of admitting (functorial) verifications of some exactness sequents (of type~$\varnothing$) $\alpha\vdash\beta$ with $\alpha$ being unconditional of finite kind and $\beta$ being constructible. With the exactness sequent represented by
$$\begin{array}{c|ccl}
\AA && \B & \\ &  && \\ \varnothing & 1 && \text{where }1 \text{ represents the terminal object,} 
\end{array}$$
one describes the property of having a terminal object. The exactness sequent represented by
$$\begin{array}{c|ccl}
\AA && \B & \\ &  && \\
\vcenter{\cd{&Y \ar[d]^-{g} \\ X \ar[r]_-{f} & Z }} & \vcenter{\cd{P \ar[d]_-{p_1} \ar[r]^-{p_2} & Y \ar[d]^-{g} \\ X \ar[r]_-{f}  & Z }} && \begin{array}{l} \text{conditions from }\A\text{ together with:} \\ (P,p_1,p_2) \text{ represents the pullback of } f \text{ along }g,\end{array}
\end{array}$$
expresses the property of having pullbacks. It remains to find an exactness sequent satisfying the required properties and describing a property which is equivalent, in the presence of finite limits, to the property of having coequalizers of kernel pairs and pullback stable regular epimorphisms. This can be done via the subsketch inclusion represented below.
$$\begin{array}{c|c}
\AA & \B \\ &  \\ \\ \cd{\\&& Y \ar[dd]^-{g} \\ \\ X \ar[rr]_-{f} && Z} &
\cd{&& P \ar[rd]^-{i'} \ar[dd]_-{g'} & \\S \ar@<4pt>[r]^-{s_1} \ar@<-3pt>[r]_-{s_2} & P' \ar[ru] \ar@{}@<-3pt>[ru]^(.5){p'} \ar[dd]_-{g''} && Y \ar[dd]^-{g} \\  && I \ar[rd]^-{i}& \\ R \ar@<4pt>[r]^-{r_1} \ar@<-3pt>[r]_-{r_2} & X \ar[rr]_-{f} \ar[ru]^-{p} && Z} \\ \\
&  \\
& \begin{array}{l}
\text{conditions from }\A\text{ together with:} \\
(R,r_1,r_2) \text{ represents the kernel pair of }f,\\
(I,p) \text{ represents the coequalizer of } r_1 \text{ and } r_2,\\
i \cdot p = f,\\
(P,g',i') \text{ represents the pullback of } i \text{ along } g,\\
(P',g'',p') \text{ represents the pullback of } p \text{ along } g',\\
(S,s_1,s_2) \text{ represents the kernel pair of }p',\\
(P,p') \text{ represents the coequalizer of } s_1 \text{ and } s_2.
\end{array}
\end{array}$$
We notice that $\beta$ is not here formally constructible; but adding the (trivial) commutativity conditions $f\cdot r_1=h$, $f\cdot r_2=h$ and $i\cdot k=h$ to $\B$ (where $h\colon R\to Z$ and $k\colon R\to I$ are the omitted arrows coming respectively from the conditions `$(R,r_1,r_2)$ represents the kernel pair of $f$' and `$(I,p)$ represents the coequalizer of $r_1$ and $r_2$') will turn it into a constructible one.
\end{example}

\begin{example}
Being a linear category (see e.g.~\cite{lawshan}) is also equivalent to the property of admitting (functorial) verifications of some exactness sequents (of type $\varnothing$) $\alpha\vdash\beta$ with $\alpha$ being unconditional of finite kind and $\beta$ being constructible. For this, we first consider the property of having a zero object, described via the following exactness sequent.
$$\begin{array}{c|ccl}
\AA && \B & \\ \varnothing & 0 && \text{where } 0 \text{ represents the terminal object, and} \\ &&& 0 \text{ represents the initial object.} 
\end{array}$$
Actually, we just need this property to ensure the category is not empty. In view of Definition~1.10.1 in~\cite{BB}, it then remains to consider the exactness sequent displayed below.
$$\begin{array}{c|c}
\AA & \B \\ &  \\ \\ \cd{X \\ \\ Y} &
\cd{&&&& X \ar@<-5pt>@/_/[lllld]_(.7){w_X} \ar@<-2pt>[dd]_(.3){z}|!{[dll];[drr]}\hole \ar[lld]_(.6){i_X} \ar@<1pt>[rrd]_-{l_X} && \\ 0 \ar@<1pt>@/^/[rrrru]_(.3){a_X} \ar@<-1pt>@/_/[rrrrd]^(.3){a_Y} && C \ar[rrrr]^(.3){f} &&&& P \ar@<-5pt>[llu]_-{p_X} \ar@<5pt>[lld]^-{p_Y} \\ &&&& Y \ar@<5pt>@/^/[llllu]^(.7){w_Y} \ar@<-2pt>[uu]_(.7){z'}|!{[ull];[urr]}\hole  \ar[llu]^(.6){i_Y} \ar@<-1pt>[rru]^-{r_Y} &&} \\ \\
&  \\ \\
& \begin{array}{l}
\text{conditions from }\A\text{ together with:}\\
0\text{ represents the terminal object},\\
0\text{ represents the initial object},\\
z = a_Y \cdot w_X,\\
z' = a_X \cdot w_Y,\\
(P,p_X,p_Y)\text{ represents the product of }X\text{ and }Y,\\
p_X \cdot l_X = 1_X,\\
p_Y \cdot l_X = z,\\
p_X \cdot r_Y = z',\\
p_Y \cdot r_Y = 1_Y,\\
(C,i_X,i_Y)\text{ represents the coproduct of }X\text{ and }Y,\\
f \cdot i_X = l_X,\\
f \cdot i_Y = r_Y,\\
f\colon C\to P\text{ represents an isomorphism}.
\end{array}
\end{array}$$
\end{example}

Other properties of a category can also be expressed as the property of admitting (functorial) verifications of some exactness sequents (of type $\varnothing$) $\alpha\vdash\beta$ with $\alpha$ being unconditional of finite kind and $\beta$ being constructible. We give here a (non-exhaustive) list of such properties:
\begin{itemize}
\item having limits of shape $\AA$, for a finite category $\AA$,
\item having colimits of shape $\AA$, for a finite category $\AA$,
\item being a groupoid,
\item being a preorder,
\item having a zero object,
\item being a regular category with $(M,X)$-closed relations~\cite{janelidze5}, for an extended matrix $(M,X)$ of terms in the algebraic theory of sets; this includes the examples of
\begin{itemize}
\item being an $n$-permutable category (for a fixed $n \geqslant 2$)~\cite{CKP}, and so by taking $n=2$, being a regular Mal'tsev category~\cite{CLP},
\item being a regular majority category~\cite{hoefnagel},
\end{itemize}
\item being a regular pointed category with $(M,X)$-closed relations~\cite{janelidze5}, for an extended matrix $(M,X)$ of terms in the algebraic theory of pointed sets; this includes the examples of
\begin{itemize}
\item being regular unital~\cite{bourn2},
\item being regular strongly unital~\cite{bourn2},
\item being regular subtractive~\cite{janelidze},
\end{itemize}
\item being a Barr-exact Mal'tsev category~\cite{barr,CLP},
\item being regular protomodular with binary coproducts~\cite{bourn},
\item being regular and having involution-rigidness property with binary coproducts~\cite{janmar},
\item being weakly Mal'tsev with binary coproducts~\cite{MF},
\item being semi-abelian~\cite{JMT},
\item being abelian~\cite{maclane book,buchsbaum,Gro57,maclane},
\item being additive, see e.g.~\cite{maclane book},
\item being normal~\cite{janelidzenormal},
\item being semi-abelian with the `Smith is Huq' condition~\cite{BG},
\item being semi-abelian with the `normality of Higgins commutators' condition~\cite{cigolithesis,CGVdL},
\item being an algebraically coherent semi-abelian category~\cite{CGVdL2},
\item being coherent with finite coproducts~\cite{johnstone,MR},
\item being distributive~\cite{CLW},
\item being extensive with pullbacks~\cite{CLW}.
\end{itemize}
In view of Subsection~\ref{section dual sequents}, the dual properties of all these could be added to the list. The case of being a regular category with $(M,X)$-closed relations is treated as Example~3.16 in~\cite{jacqminthesis}. For the property of being a Barr-exact Mal'tsev category, we use the fact that this is equivalent to being a regular Mal'tsev category with the additional property that for each reflexive graph
$$\cd{X \ar@<5pt>[r]^-{f} \ar@<-2pt>[r]_-{g} & Y \ar@/^1.2pc/@<1pt>[l]^-{s}}$$ the factorisation $p$ of $f$ and $g$ through the kernel pair of their coequalizer must be a regular epimorphism (i.e., the coequalizer of its kernel pair).
$$\cd{R_q \ar@{}[r]|(.14){}="A" \ar@<8pt>@/^/"A";[rd] \ar@<20pt>@{}[rd]|(.4){r_1} \ar@<3pt>@/^/[rd] \ar@<3pt>@{}[rd]|(.4){r_2} && \\ X \ar[u]^-{p} \ar@<5pt>[r]^-{f} \ar@<-2pt>[r]_-{g} & Y \ar@/^1.2pc/@<1pt>[l]^-{s} \ar@{->>}[r]^-{q} & Q}$$
The property involving regular protomodularity is explicitly described just before Proposition~4.27 in~\cite{jacqminthesis}. As remarked there, the assumption of the existence of binary coproducts can be replaced by the assumption of the existence of pushouts of morphisms along split monomorphisms. Similarly, we can get the case of involution-rigidness with binary coproducts by Theorem~3.2 in~\cite{janmar} and the case of weakly Mal'tsev categories with binary coproducts from its definition.
For semi-abelian categories, it suffices to present semi-abelianness as having a zero object, being a Barr-exact Mal'tsev category and being regular protomodular with binary coproducts.
An easy way to describe the abelian case is now to say that a category is abelian if and only if it is semi-abelian and its dual is also semi-abelian~\cite{JMT} (see Subsection~\ref{section dual sequents}).
The additive case is clear from Theorem~1.10.14 in~\cite{BB}. The case of normal categories follows directly from the definition. One could have also included the property of being a finitely complete pointed category where every split epimorphism is normal, a property which is equivalent to normality in the regular pointed context according to Theorem~4.0.3 in~\cite{BJ2}. The `Smith is Huq' condition follows from Theorem~4.6 in~\cite{HVdL} and the `normality of Higgins commutators' condition is evident from its definition. For the example of algebraically coherent semi-abelian categories, one could use Proposition~3.13 in~\cite{CGVdL2}. The example of being a coherent category with finite coproducts is easy to show from the definition (see~\cite{johnstone} and references therein) once one has remarked that in a regular category with finite coproducts, the union of two subobjects $s \colon S \rightarrowtail A$ and $t \colon T \rightarrowtail A$ is given by the image of the factorisation $\left( \begin{smallmatrix} s \\ t \end{smallmatrix} \right) \colon S+T \rightarrow A$.
The distributive case being obvious, let us finally discuss the example of extensive categories with pullbacks. By Proposition~2.2 in~\cite{CLW}, a category with pullbacks and binary coproducts is extensive if and only if for any commutative diagram
$$\cd{A_1 \ar[d]_-{f_1} \ar[r]^-{a_1} & A \ar[d]_-{f} & A_2 \ar[l]_-{a_2} \ar[d]^-{f_2} \\ X_1 \ar[r]_-{x_1} & X_1+X_2 & X_2 \ar[l]^-{x_2}}$$
where the bottom row is a coproduct diagram, the two squares are simultaneously pullbacks exactly when the top row is a coproduct diagram.
In one direction, this means that starting from $f_i \colon A_i \rightarrow X_i$ (for $i \in \{1,2\}$), we require the square
$$\cd{A_i \ar[d]_-{f_i} \ar[r]^-{a_i} & A_1+A_2 \ar[d]^-{f_1+f_2} \\ X_i \ar[r]_-{x_i} & X_1+X_2}$$
to be a pullback for each $i \in \{1,2\}$. The converse implication can be expressed as: given three morphisms with the same codomain,
$$\cd{&A \ar[d]^-{f} & \\ X_1 \ar[r]_-{y} & X & X_2 \ar[l]^-{z}}$$
considering the pullback of $f$ along the induced morphism $\left( \begin{smallmatrix} y \\ z \end{smallmatrix} \right)\colon X_1+X_2 \to X$,
$$\cd{P \ar[r]^-{w} \ar[d]_-{f'} \ar@{}[rd]|-{\textrm{pb}} & A \ar[d]^-{f} \\ X_1+X_2 \ar[r]_-{\left( \begin{smallmatrix} y \\ z \end{smallmatrix} \right)} & X}$$
then the pullbacks of the coproduct injections along $f'$ give rise to a coproduct diagram.
$$\cd{P_1 \ar[d]_-{f_1} \ar[r]^-{x'_1} \ar@{}[rd]|-{\textrm{pb}} & P \ar[d]_-{f'} & P_2 \ar[l]_-{x'_2} \ar[d]^-{f_2} \ar@{}[ld]|-{\textrm{pb}} \\ X_1 \ar[r]_-{x_1} & X_1+X_2 & X_2 \ar[l]^-{x_2}}$$

The following is a list of further examples of properties of admitting (functorial) verifications of some exactness sequents $\alpha\vdash\beta$ for which either $\alpha$ is not unconditional of finite kind or $\beta$ is not constructible:
\begin{itemize}
\item The axiom of choice on a category $\CC$, stating that every epimorphism is a split epimorphism, can be expressed as a property of the form $\alpha\vdash_\CC\beta$ where neither $\alpha$ is unconditional of finite kind, nor $\beta$ is constructible.

\item For every endomorphism in a category $\CC$ to be idempotent is an exactness property of the form $\alpha\vdash_\CC\beta$ where $\alpha$ is not unconditional of finite kind, while $\beta$ is constructible.
    
\item For a commutative algebraic theory $\mathcal{T}$, a $\mathcal{T}$-enrichment of a category $\CC$ with finite products is a functorial verification of an $\varnothing$-sequent $\alpha\vdash\beta$ for $\CC$ (see e.g.~\cite{KF,janelidze1}), where $\alpha$ is unconditional of finite kind, but $\beta$, in general, is not constructible. These sequents may admit several functorial verifications for the same $\CC$, although usually they are unique.
    
\item Consider a morphism $t\colon 1\to\Omega$ in a category $\CC$, where $1$ is a terminal object $\CC$. The property for $\Omega$ to be a subobject classifier with $t$ as the truth morphism (see e.g.~\cite{johnstone}) can be expressed as the existence of verifications of two exactness sequents for the same structure.
\end{itemize}

\section{The stability theorem}

If $\CC$ is a small finitely complete category,
we denote by $\LC$ (or by $\CCC$ interchangeably, following~\cite{barr2}) the dual of the category of finite limit preserving functors from $\CC$ to~$\Set$. We will consider the (restricted) Yoneda embedding $$\y \colon \CC \hookrightarrow \LC,\quad C\mapsto \hom_\CC(C,-)$$ which fully embeds $\CC$ in $\LC$. As shown in~\cite{AGV,GU}, this embedding is the free cofiltered limit completion of $\CC$. Furthermore, we have: 

\begin{theorem} \cite{AGV,GU} \label{c tilde lemma}
For any finitely complete small category $\CC$, we have:
\begin{enumerate}
\item \label{bicompleteness}  $\CCC$ is complete and cocomplete.
\item  The embedding $\y \colon \CC \hookrightarrow \CCC$ is fully faithful and thus reflects isomorphisms.
\item \label{i preserves} The embedding $\y \colon \CC \hookrightarrow \CCC$ preserves (and reflects) finite limits and all colimits.
\setcounter{c tilde enumeration}{\value{enumi}}
\end{enumerate}
\end{theorem}

Moreover, viewing the different hom-sets of $\CC$ as pairwise disjoint sets, we will assume that $\y$ is also injective on objects. More technical properties of $\y$ will be recalled in Theorem~\ref{c tilde lemma 2}. By Theorem~\ref{c tilde lemma}(\ref{i preserves}), for a small category $\CC$ with finite limits, if $S$ is a $\Z$-structure in $\CC$ then $\y S=\y\circ S$ is a $\Z$-structure in $\CCC$, for any sketch $\Z$. This defines a functor
$$\y_\Z\colon\Z\CC\to\Z\CCC.$$

We are now ready to formulate our `stability theorem':

\begin{theorem}\label{C tilde preserves unconditional exactness properties} Consider an exactness sequent $\cd{\X \ar[r]^-{\alpha} & \A \ar[r]^-{\beta} & \B}$ such that $\alpha$ is unconditional of finite kind,
and let $\CC$ be a small finitely complete category. If there is a functorial verification of $\alpha\vdash\beta$ for an $\X$-structure $F$ in $\CC$, then there is also a functorial verification of $\alpha\vdash\beta$ for $\y F$ in~$\LC$. In other words, $$\alpha\vdash_F\beta \text{ functorially $\Rightarrow$ }\alpha\vdash_{\y F}\beta \text{ functorially}$$
for any $\X$-structure $F$ in $\CC$.
\end{theorem}

Let us make explicit that in the case where $\X=\varnothing$ is the empty sketch, this theorem shows that the pro-completion $\LC$ inherits many exactness properties from the category~$\CC$. In other words, using the notation for the case $\X=\varnothing$ from Subsection~\ref{secB}, we immediately get the following corollary.

\begin{corollary}\label{corollary stability theorem} Consider an $\varnothing$-sequent $\cd{\varnothing \ar[r]^-{\alpha} & \A \ar[r]^-{\beta} & \B}$ such that $\alpha$ is unconditional of finite kind,
and let $\CC$ be a small finitely complete category. If there is a functorial verification of $\alpha\vdash\beta$ for~$\CC$, then there is also one for~$\LC$. In other words, $$\alpha\vdash_{\CC}\beta \text{ functorially $\Rightarrow$ }\alpha\vdash_{\LC}\beta \text{ functorially.}$$
\end{corollary}

For a small finitely cocomplete category $\CC$, we denote by $\y'\colon\CC\hookrightarrow\Lex(\CC^{\op},\Set)$ the (restricted) Yoneda embedding $C\mapsto\hom_\CC(-,C)$. By applying Theorem~\ref{C tilde preserves unconditional exactness properties} with the exactness sequent $\cd{\X^{\op} \ar[r]^-{\alpha^{\op}} & \A^{\op} \ar[r]^-{\beta^{\op}} & \B^{\op}}$, the category $\CC^{\op}$ and the $\X^{\op}$-structure $F^{\op}$ in~$\CC^{\op}$, we get the following dual formulation of the stability theorem:

\begin{theorem}\label{dual stability theorem} Consider an exactness sequent $\cd{\X \ar[r]^-{\alpha} & \A \ar[r]^-{\beta} & \B}$ such that
$\alpha$ is unconditional of finite kind, and let $\CC$ be a small finitely cocomplete category. If there is a functorial verification of $\alpha\vdash\beta$ for an $\X$-structure $F$ in $\CC$, then there is also a functorial verification of $\alpha\vdash\beta$ for $\y' F$ in~$\Lex(\CC^{\op},\Set)$. In other words, $$\alpha\vdash_F\beta \text{ functorially $\Rightarrow$ }\alpha\vdash_{\y' F}\beta \text{ functorially}$$
for any $\X$-structure $F$ in $\CC$.
\end{theorem}

Subsection~\ref{section examples} provides many examples of exactness properties to which Theorems~\ref{C tilde preserves unconditional exactness properties} and~\ref{dual stability theorem} can be applied. As mentioned in the Introduction, several particular instances of these theorems can already be found in the literature for particular exactness properties of a category, see \cite{barr2,BP,CPR,DS,DS2,GP}. The property of being cartesian closed has also been proved~\cite{DS2} to transfer from a small finitely cocomplete category $\CC$ to $\Lex(\CC^{\op},\Set)$, but we have not been able to deduce this fact from Theorem~\ref{dual stability theorem}, which perhaps suggests that our theorem could be further generalized. The property for a morphism $t\colon1\to\Omega$ to be the truth morphism representing $\Omega$ as a subobject classifier was claimed in~\cite{DS2} to be transferred from a regular finitely cocomplete small category $\CC$ to $\Lex(\CC^{\op},\Set)$. However, in~\cite{BP}, it has been shown that this stability result is false. This is not surprising from the point of view of our stability theorem. The above property for $t$ can be expressed as the existence of verifications of exactness sequents $\alpha\vdash\beta$ with $\alpha$ being unconditional of finite kind. However, in this case, this is not equivalent to the existence of functorial verifications, since the $\beta$'s are not all constructible.

As we can see from Subsection~\ref{section examples}, the pro-completion of any regular subtractive small category is again a regular subtractive category by Corollary~\ref{corollary stability theorem}. Under the axiom of universes, this resolves positively the open question from~\cite{boujan abelianization} whether any regular subtractive category admits an embedding to one with binary coproducts.

Let $\CC$ be the full subcategory of $\Set$ consisting of subsets of $\NN$, the set of natural numbers. This category is small, finitely cocomplete and Barr-exact. However, according to~\cite{BP,CPR}, since $\CC$ is not `pro-exact', the category $\Lex(\CC^{\op},\Set)$ is not Barr-exact. In view of Theorem~\ref{dual stability theorem}, this shows that being a Barr-exact category cannot be presented as the property of admitting functorial verifications of some exactness sequents $\alpha\vdash\beta$ with every $\alpha$ being unconditional of finite kind. However, it is not difficult to see that it can be written as the property of admitting functorial verifications of some exactness sequents $\alpha\vdash\beta$ with constructible $\beta$'s.

We conclude this section with a remark about the converse of the stability theorem. Because of the properties of the functor $\y$, we have the implication  
$$\alpha\vdash_{\y F}\beta \text{ functorially $\Rightarrow$ }\alpha\vdash_F\beta \text{ functorially}$$
at least in the following two cases:
\begin{itemize}
\item when $\alpha$ is any subsketch inclusion, $\beta$ is constructible and $\CC$ is finitely cocomplete (in addition to being small finitely complete),
\item when $\alpha$ is any subsketch inclusion, $\beta$ is constructible and $\B$ does not contain any colimit condition (and $\CC$ is small finitely complete).
\end{itemize}
Indeed, in those cases, if $G$ is an $\A$-structure in $\CC$ such that $G\circ\alpha = F$, we know that $\y G$ is an $\A$-structure in $\Lex(\CC^{\op},\Set)$ such that $(\y G)\circ \alpha = \y F$. Suppose that $\alpha\vdash_{\y F}\beta$ holds (functorially). Then there is a $\B$-structure $H'$ in $\Lex(\CC^{\op},\Set)$ such that $H' \circ \beta = \y G$. Now, in view of the following, we can show by induction that $H'$ lies, up to isomorphism, in the image of~$\y_{\B}$: 
\begin{itemize}
\item the step-by-step construction of $\beta$ (Subsection~\ref{section constructible sequents}), 

\item in~$\CC$, we can construct composites of paths, finite limits, morphisms induced by them and, in the first case, finite colimits (trivial),

\item the functor $\y$ preserves and reflects commutativity, finite limits and finite colimits (Theorem~\ref{c tilde lemma}).
\end{itemize}
One then concludes by Lemma~\ref{lemma verification constructible}.

\section{Preliminaries for the proof of the stability theorem}\label{section preliminaries}

In this section we recall some well-known concepts and facts and also fix notation as a preparation for the proof of the stability theorem.

\subsection{2-categorical notation}

Functors between categories are represented by a single arrow as in $M\colon \AA\to\BB$. The composite of functors $M\colon \AA\to\BB$ and $N\colon \BB\to\CC$ is denoted by $N\circ M$. Natural transformations are represented by double arrows as in $\mathfrak{m}\colon M\Rightarrow M'$. The vertical composite of $\mathfrak{m}\colon M\Rightarrow M'$ and $\mathfrak{m'}\colon M'\Rightarrow M''$ is denoted by $\mathfrak{m'}\circ \mathfrak{m}$. The horizontal composite of $\mathfrak{m}$ and $\mathfrak{n}$ in a display
$$\cd{\AA \ar@/^1pc/[rr]^-{M} \ar@{}@<10pt>[rr]^-{}="M" \ar@/_1pc/[rr]_-{M'} \ar@{}@<-10pt>[rr]_-{}="Y" \ar@{=>}"M";"Y"^-{\,\mathfrak{m}} && \BB \ar@/^1pc/[rr]^-{N} \ar@{}@<10pt>[rr]^-{}="N" \ar@/_1pc/[rr]_-{N'} \ar@{}@<-10pt>[rr]_-{}="Z" \ar@{=>}"N";"Z"^-{\,\mathfrak{n}} && \CC}$$
is denoted by $\mathfrak{n}\bullet\mathfrak{m}\colon N\circ M\Rightarrow N'\circ M'$; recall that it is defined by $$(\mathfrak{n}\bullet\mathfrak{m})_A=\mathfrak{n}_{M'(A)} \circ N(\mathfrak{m}_A) = N'(\mathfrak{m}_A) \circ \mathfrak{n}_{M(A)}$$ for each object $A$ of $\AA$. As usual, for the sake of brevity, $1_N \bullet \mathfrak{m}$ is abbreviated as $N \bullet \mathfrak{m}$ and $\mathfrak{n} \bullet 1_M$ is abbreviated as $\mathfrak{n} \bullet M$.
The `middle interchange law' says that given a diagram
$$\cd{\AA \ar@<4pt>@/^1.4pc/[rr]^-{M} \ar@{}@<19pt>[rr]^-{}="M" \ar[rr]|-{M'} \ar@{}@<2pt>[rr]_-{}="Y" \ar@{=>}"M";"Y"^-{\,\mathfrak{m}} \ar@<-4pt>@/_1.4pc/[rr]_-{M''} \ar@{}@<-4pt>[rr]^-{}="F" \ar@{}@<-19pt>[rr]_-{}="A" \ar@{=>}"F";"A"^<<{\,\mathfrak{m'}} && \BB \ar@<4pt>@/^1.4pc/[rr]^-{N} \ar@{}@<19pt>[rr]^-{}="N" \ar[rr]|-{N'} \ar@{}@<2pt>[rr]_-{}="Z" \ar@{=>}"N";"Z"^-{\,\mathfrak{n}} \ar@<-4pt>@/_1.4pc/[rr]_-{N''} \ar@{}@<-4pt>[rr]^-{}="G" \ar@{}@<-19pt>[rr]_-{}="B" \ar@{=>}"G";"B"^<<{\,\mathfrak{n'}} && \CC}$$
of functors and natural transformations, the equality
$$(\mathfrak{n'} \circ \mathfrak{n}) \bullet (\mathfrak{m'} \circ \mathfrak{m}) = (\mathfrak{n'} \bullet \mathfrak{m'}) \circ (\mathfrak{n} \bullet \mathfrak{m})$$
holds. We can represent identities involving vertical and horizontal composition of natural transformations as `pasting identities' using the symbol `$\equiv$' between the corresponding diagrams. For example, the pasting representation of the previous identity is:
$$\vcenter{\cd{\AA \ar@/^1.3pc/[rr]^-{M} \ar@{}@<13pt>[rr]^-{}="M" \ar@/_1.3pc/[rr]_-{M''} \ar@{}@<-13pt>[rr]_-{}="Y" \ar@{=>}"M";"Y"^<<<<{\mathfrak{m'}\circ\mathfrak{m}} && \BB \ar@/^1.3pc/[rr]^-{N} \ar@{}@<13pt>[rr]^-{}="N" \ar@/_1.3pc/[rr]_-{N''} \ar@{}@<-13pt>[rr]_-{}="Z" \ar@{=>}"N";"Z"^<<<<{\mathfrak{n'}\circ\mathfrak{n}} && \CC}} \qquad \equiv \qquad \vcenter{\cd{\AA \ar@/^2.3pc/[rrrr]^-{N\circ M} \ar@{}@<25pt>[rrrr]^-{}="M" \ar[rrrr]|-{N'\circ M'} \ar@{}@<4pt>[rrrr]_-{}="Y" \ar@{=>}"M";"Y"^-{\,\mathfrak{n}\bullet\mathfrak{m}} \ar@/_2.3pc/[rrrr]_-{N''\circ M''} \ar@{}@<-4pt>[rrrr]^-{}="F" \ar@{}@<-25pt>[rrrr]_-{}="A" \ar@{=>}"F";"A"^-{\,\mathfrak{n'}\bullet\mathfrak{m'}} &&&& \CC}}$$

\subsection{The category $\1$}

We write $\1$ for the single-morphism category. For a category $\AA$, by $!_\AA$ we denote the unique functor $\AA\to\1$. Since this functor is uniquely determined by its domain and codomain, when both of these are displayed in a diagram, we do not include the label `$!_\AA$' for the corresponding arrow. 
Furthermore, we do not distinguish between an object $A\in\AA$ and the functor $\1\to\AA$ which maps the unique morphism of $\1$ to the identity morphism of $A$.

\subsection{Cones}

Given a functor $M\colon \AA\to\BB$ and an object $B\in\BB$, we view a cone over $M$, with vertex $B$, as a natural transformation
$$\cd{ & \1\ar[dr]^-{B} & \\ \AA \ar[rr]_-{M}\ar[ur]\urrtwocell<\omit>{<1>} &  & \BB.}$$ Given two cones $\mathfrak{b}_1 \colon B_1\circ !_\AA \Rightarrow M$ and $\mathfrak{b}_2\colon B_2\circ !_\AA \Rightarrow M$ over $M$, a morphism of cones $m\colon \mathfrak{b}_1\to \mathfrak{b}_2$ is a morphism $m\colon B_1\to B_2$ in $\BB$ such that for the corresponding natural transformation $m \colon B_1 \Rightarrow B_2$, the pasting identity
$$\vcenter{\cd{ & \1\ar[dr]|(.32){\hole}|(.35){B_2}|(.4){\hole} \ar@/^25pt/[dr]|-{B_1}|(.55){\hole}\drtwocell<\omit>{<-2>\,\,m} & \\ \AA \ar[rr]_-{M}\ar[ur]\urrtwocell<\omit>{<1>\;\;\mathfrak{b_2}} &  & \BB }}\quad\quad\equiv\quad\vcenter{\cd{ & \1\ar[dr]^-{B_1} & \\ \AA \ar[rr]_-{M}\ar[ur]\urrtwocell<\omit>{<1>\;\;\mathfrak{b_1}} &  & \BB }}$$
holds. A limiting cone over a functor $M$ is then a cone 
$$\cd@=40pt{ & \1\ar[dr]^-{\lim M} & \\ \AA \ar[rr]_-{M}\ar[ur]\urrtwocell<\omit>{<1>\quad\mathfrak{p}^M} &  & \BB }$$ 
over $M$ such that for any other cone $\mathfrak{b}$ over $M$, there is a unique morphism of cones $\mathfrak{b}\to\mathfrak{p}^M$. 

\subsection{Comma categories}

Given a functor $M\colon \AA\to\BB$ and an object $B$ of $\BB$, the usual comma category $(B\downarrow M)$ comes equipped with the data
$$\cd@=30pt{(B\downarrow M)\ar[d]_-{B^{\downarrow M}}\ar[r]\drtwocell<\omit>{<1>\qquad(B,M)} & \1\ar[d]^-{B}\\ \AA\ar[r]_-{M} & \BB }$$
which is universal among such data: for any triple $(\XX,X,\mathfrak{b})$
$$\cd@=30pt{\XX\ar[d]_-{X}\ar[r]\drtwocell<\omit>{<0> \mathfrak{b}} & \1\ar[d]^-{B}\\ \AA\ar[r]_-{M} & \BB }$$
there exists a unique functor $U\colon\XX\to (B\downarrow M)$ such that $B^{\downarrow M}\circ U=X$ and the pasting identity
$$\vcenter{\cd@=30pt{ & (B\downarrow M)\ar[d]^-{B^{\downarrow M}}\ar[r]\drtwocell<\omit>{<1>\quad\quad(B,M)} & \1\ar[d]^-{B}\\ \XX\ar@{}[urr]|<<<<<<<<<<<<<{=}\ar[r]_-{X}\ar[ur]^-{U} & \AA\ar[r]_-{M} & \BB }}\quad\equiv\quad\vcenter{\cd@=30pt{\XX\ar[d]_-{X}\ar[r]\drtwocell<\omit>{<0> \mathfrak{b}} & \1\ar[d]^-{B}\\ \AA\ar[r]_-{M} & \BB }}$$
holds, and furthermore, for any two functors $U,V\colon \XX\to(B\downarrow M)$ and a natural transformation $\mathfrak{n}\colon B^{\downarrow M}\circ U\Rightarrow B^{\downarrow M}\circ V$ such that 
$$(M\bullet \mathfrak{n})\circ((B,M)\bullet U)=(B,M)\bullet V,$$
there exists a unique natural transformation $\mathfrak{m}\colon U\Rightarrow V$ such that 
$$B^{\downarrow M}\bullet \mathfrak{m}=\mathfrak{n}.$$
The second part of the universal property implies that $B^{\downarrow M}$ is faithful and if $\mathfrak{n}\colon B^{\downarrow M}\circ U\Rightarrow B^{\downarrow M}\circ V$ is a natural isomorphism, then so is $\mathfrak{m}$ in the above notation.

\subsection{Pullbacks of functors are strong}

Given a pullback of functors,
$$\cd{\PP \ar[r]^-{M'} \ar[d]_-{N'} \ar@{}[dr]|{\textrm{pb}} & \BB \ar[d]^-{N} \\ \AA \ar[r]_-{M} & \DD}$$
for any two functors $U,V\colon \XX\to\PP$ and natural transformations $\mathfrak{n}\colon N'\circ U\Rightarrow N'\circ V$ and $\mathfrak{m}\colon M'\circ U\Rightarrow M'\circ V$ satisfying
$$M\bullet \mathfrak{n}=N\bullet \mathfrak{m},$$
there exists a unique natural transformation $\mathfrak{q}\colon U\Rightarrow V$ such that 
$$N'\bullet \mathfrak{q}=\mathfrak{n} \qquad \text{and} \qquad M' \bullet \mathfrak{q}=\mathfrak{m}.$$
Following the terminology in~\cite{JMMV}, we say that such pullbacks are \emph{strong}. This property implies in particular that if $\mathfrak{n}$ and $\mathfrak{m}$ are natural isomorphisms in the above notation, then so is $\mathfrak{q}$.

\subsection{Pro-completion and the evaluation functors}\label{subsection procompletion}

One can extend Theorem~\ref{c tilde lemma} with the following properties of~$\CCC$:

\begin{theorem} \cite{AGV,GU} \label{c tilde lemma 2}
For any finitely complete small category $\CC$, we have:
\begin{enumerate}
\setcounter{enumi}{\value{c tilde enumeration}}
\item \label{commutativity of limits} In $\CCC$, cofiltered limits commute with limits and finite colimits.
\item \label{cofiltered limit} The cone $(P,\y)$ over the functor $\y \circ {P^{\downarrow\y}}$ is a limiting cone, for each object $P\in\CCC$. 
\end{enumerate}
$$\cd@=30pt{(P\downarrow \y)\ar[d]_-{P^{\downarrow \y}}\ar[r]\drtwocell<\omit>{<1>\quad\,\,\,(P,\y)} & \1\ar[d]^-{P}\\ \CC\ar[r]_-{\y} & \CCC }$$
\end{theorem}

Let $\Z$ be an arbitrary sketch and $\mathcal{W}$ a sketch with one object, no arrows and no conditions. Objects $Z$ in $\Z$ correspond bijectively to morphisms of sketches $Z\colon\mathcal{W}\to\Z$. Given a finitely complete small category $\CC$, the corresponding functors $Z_{\CC}\colon \Z\CC\to\CC$ and $Z_{\CCC}\colon \Z\CCC\to\CCC$ are the usual `evaluation functors' at the object $Z$, where we identified the category $\mathcal{W}\CC$ with $\CC$ and the category $\mathcal{W}\CCC$ with~$\CCC$. By Theorems~\ref{c tilde lemma}(\ref{bicompleteness}) and~\ref{c tilde lemma 2}(\ref{commutativity of limits}), the category $\Z\CCC$ has cofiltered limits and they can be computed component-wise. This implies that the family $(Z_{\CCC})_{Z\in\Z}$ of evaluation functors preserves and jointly reflects cofiltered limits. 

\subsection{Pro-completion and Kan extensions}\label{subsection kan extensions}

If $\mu\colon \V\to\Z$ represents a morphism of sketches and if $\CC$ is still a finitely complete small category, the diagram
$$\cd{ \Z\CC\ar[d]_-{\mu_\CC}\ar[r]^-{\y_\Z} & \Z\CCC\ar[d]^-{\mu_{\CCC}} \\ \V\CC\ar[r]_-{\y_\V} & \V\CCC}$$
obviously commutes. When $\V$ and $\Z$ are underlying sketches of finite categories, $\V\CC$ and $\Z\CC$ become the usual functor categories and via point-wise (right) Kan extension along~$\mu$, the functors $\mu_\CC$ and $\mu_{\CCC}$ both obtain right adjoints. Moreover, since $\y$ preserves finite limits, the Beck-Chevalley condition holds for the above square. In fact, since $\y$ is also injective on objects and faithful, once a right adjoint $\mu'_\CC$ for $\mu_\CC$ is chosen, we can choose a right adjoint $\mu'_{\CCC}$ for $\mu_{\CCC}$ in such a way that the Beck-Chevalley condition holds strictly. In other words, $\y_\Z \circ\mu'_\CC=\mu'_{\CCC}\circ\y_\V $ and when $\mathfrak{y}$ and $\mathfrak{e}$ denote the unit and the counit for the adjunction $\mu_\CC\dashv\mu'_\CC$, while $\mathfrak{y}'$ and $\mathfrak{e}'$ for the adjunction $\mu_{\CCC}\dashv\mu'_{\CCC}$, we have: $\y_\Z \bullet\mathfrak{y}=\mathfrak{y}'\bullet\y_\Z $ and $\y_\V \bullet\mathfrak{e}=\mathfrak{e}'\bullet\y_\V $. If in addition, as a functor between categories, $\mu$ is fully faithful and injective on objects, we can choose these adjunctions in such a way that counits for both adjunctions are identity natural transformations, and so the right adjoints $\mu'_\CC$ and $\mu'_{\CCC}$ are at the same time right inverses of $\mu_\CC$ and $\mu_{\CCC}$ respectively.

If $\mu\colon\V\to\Z$ is any morphism between underlying sketches of finite categories, if $\V'$ is formed from $\V$ by adding a set $X_1$ of commutativity conditions to $\V$, if $\Z'$ is obtained by adding to $\Z$ a set $X_2$ of commutativity conditions containing the images under $\mu$ of those in $X_1$ and if $\nu\colon\V'\to\Z'$ is the induced morphism of sketches, we can still construct right adjoints to $\nu_{\CC}$ and $\nu_{\CCC}$ which satisfy the Beck-Chevalley condition strictly. Moreover, if $\mu$ is fully faithful and injective on objects as a functor and if $X_2$ only contains the images under $\mu$ of the elements of $X_1$, we can again choose these adjunctions in such a way that both counits are identity natural transformations. To see this, it is enough to use the above arguments where we replace the finite categories inducing $\V$ and $\Z$ by their quotients under $X_1$ and $X_2$ respectively.

\subsection{Removing convergence conditions from a sketch}

For a sketch $\Z$, by $\iota^\Z\colon \Z_\ast\to \Z$ we denote the inclusion of the subsketch $\Z_\ast$ of $\Z$ obtained from $\Z$ by removing all convergence conditions in $\Z$. Furthermore, for a category $\CC$, the values of the functor $\iota^\Z_\CC\colon\Z\CC\to\Z_\ast\CC$ will be written as $\iota^\Z_\CC(S)=S_\ast$ for $\Z$-structures, and similarly, $\iota^\Z_\CC(m)=m_\ast$ for morphisms.

\section{Proof of the stability theorem}\label{section proof}

In this section we prove Theorem~\ref{C tilde preserves unconditional exactness properties}. The proof is organized in \ref{last step}~steps, excluding the following proof outline, which we have included by referee's suggestion.

\subsection*{\newstep Outline of the proof}\label{step zero}

Lemma~\ref{LemA} reduces the proof of the theorem to a functorial construction of $H_G\in\B\CCC$, from $G\in\A_{\y F}^{\alpha}\CCC$, such that there is an isomorphism $H_G \circ \beta\cong G$ natural in~$G$. In less technical language, for a given $\X$-structure $F$ in~$\X$, we want to naturally extend each $\A$-extension $G$ of $\y F$ in~$\CCC$, to a $\B$-structure $H_G$ in~$\CCC$. In Step~\ref{step definition HG} of the proof, we will obtain $H_G$ as a cofiltered limit of similar extensions $\y H$ of $\A$-extensions $\y K$ of~$\y F$, which we know exist by the assumption that $\alpha\vdash_F \beta$ and the fact that the (restricted) Yoneda embedding $\y\colon \CC\to\CCC$ preserves finite limits and all colimits.
$$
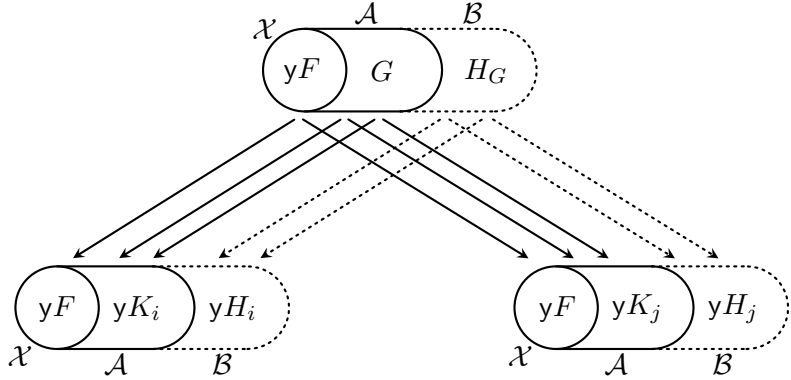
\begin{tikzpicture}[line cap=round,line join=round,>=triangle 45,x=1cm,y=1cm,scale=0.7]
\draw [line width=0.8pt] (1.3476,5.417457987852216) circle (0.7825420121477844cm);
\draw [line width=0.8pt] (1.3476,6.2)-- (3.14638,6.2);
\draw [line width=0.8pt] (1.3476,4.634915975704431)-- (3.14638,4.634915975704431);
\draw [shift={(3.14638,5.417457987852216)},line width=0.8pt]  plot[domain=-1.5707963267948966:1.5707963267948966,variable=\t]({1*0.7825420121477844*cos(\t r)+0*0.7825420121477844*sin(\t r)},{0*0.7825420121477844*cos(\t r)+1*0.7825420121477844*sin(\t r)});
\draw [line width=0.8pt,dotted] (3.14638,6.2)-- (4.923347056785893,6.2);
\draw [line width=0.8pt,dotted] (3.14638,4.634915975704431)-- (4.923347056785893,4.634915975704431);
\draw [shift={(4.923347056785893,5.417457987852216)},line width=0.8pt,dotted]  plot[domain=-1.5707963267948966:1.5707963267948966,variable=\t]({1*0.7825420121477844*cos(\t r)+0*0.7825420121477844*sin(\t r)},{0*0.7825420121477844*cos(\t r)+1*0.7825420121477844*sin(\t r)});
\draw [line width=0.8pt] (-3.30414,0.9174579878522158) circle (0.7825420121477844cm);
\draw [line width=0.8pt] (-3.30414,1.7)-- (-1.50536,1.7);
\draw [line width=0.8pt] (-3.30414,0.13491597570443137)-- (-1.50536,0.13491597570443137);
\draw [shift={(-1.50536,0.9174579878522158)},line width=0.8pt]  plot[domain=-1.5707963267948966:1.5707963267948966,variable=\t]({1*0.7825420121477844*cos(\t r)+0*0.7825420121477844*sin(\t r)},{0*0.7825420121477844*cos(\t r)+1*0.7825420121477844*sin(\t r)});
\draw [line width=0.8pt,dotted] (-1.50536,1.7)-- (0.2716070567858928,1.7);
\draw [line width=0.8pt,dotted] (-1.50536,0.13491597570443137)-- (0.2716070567858928,0.13491597570443137);
\draw [shift={(0.2716070567858928,0.9174579878522158)},line width=0.8pt,dotted]  plot[domain=-1.5707963267948966:1.5707963267948966,variable=\t]({1*0.7825420121477844*cos(\t r)+0*0.7825420121477844*sin(\t r)},{0*0.7825420121477844*cos(\t r)+1*0.7825420121477844*sin(\t r)});
\draw [line width=0.8pt] (6.07189,0.9174579878522158) circle (0.7825420121477844cm);
\draw [line width=0.8pt] (6.07189,1.7)-- (7.87067,1.7);
\draw [line width=0.8pt] (6.07189,0.13491597570443137)-- (7.87067,0.13491597570443137);
\draw [shift={(7.87067,0.9174579878522158)},line width=0.8pt]  plot[domain=-1.5707963267948966:1.5707963267948966,variable=\t]({1*0.7825420121477844*cos(\t r)+0*0.7825420121477844*sin(\t r)},{0*0.7825420121477844*cos(\t r)+1*0.7825420121477844*sin(\t r)});
\draw [line width=0.8pt,dotted] (7.87067,1.7)-- (9.647637056785893,1.7);
\draw [line width=0.8pt,dotted] (7.87067,0.13491597570443137)-- (9.647637056785893,0.13491597570443137);
\draw [shift={(9.647637056785893,0.9174579878522158)},line width=0.8pt,dotted]  plot[domain=-1.5707963267948966:1.5707963267948966,variable=\t]({1*0.7825420121477844*cos(\t r)+0*0.7825420121477844*sin(\t r)},{0*0.7825420121477844*cos(\t r)+1*0.7825420121477844*sin(\t r)});
\draw [>=stealth,->,line width=0.8pt] (2.019952933012987,4.478760669746471) -- (-2.1367855821403943,1.8561553059579605);
\draw [>=stealth,->,line width=0.8pt] (2.6585737787110286,4.478760669746471) -- (-1.4981647364423525,1.8561553059579605);
\draw [>=stealth,->,line width=0.8pt,dotted] (3.887740435778408,4.478760669746471) -- (-0.2689980793749731,1.8561553059579605);
\draw [>=stealth,->,line width=0.8pt,dotted] (4.673118057270059,4.478760669746471) -- (0.5163795421166775,1.8561553059579605);
\draw [>=stealth,->,line width=0.8pt] (2.1672215156478507,4.483374707688921) -- (6.403644126598559,1.8515412680155086);
\draw [>=stealth,->,line width=0.8pt] (2.8223578663987547,4.483374707688922) -- (7.058780477349463,1.8515412680155094);
\draw [>=stealth,->,line width=0.8pt,dotted] (4.067116932825472,4.483374707688921) -- (8.30353954377618,1.8515412680155086);
\draw [>=stealth,->,line width=0.8pt,dotted] (4.886037371264102,4.483374707688922) -- (9.12245998221481,1.8515412680155094);
\draw [>=stealth,->,line width=0.8pt] (1.1517601719839496,4.478760669746472) -- (-3.0049783431694315,1.8561553059579614);
\draw [>=stealth,->,line width=0.8pt] (1.3155442596716758,4.483374707688921) -- (5.551966870622383,1.8515412680155086);

\draw[color=black] (0.60,6.2) node {$\X$};
\draw[color=black] (2.5,6.5) node {$\A$};
\draw[color=black] (4.5,6.5) node {$\B$};
\draw[color=black] (-4,0) node {$\X$};
\draw[color=black] (-2.2,-0.15) node {$\A$};
\draw[color=black] (-0.2,-0.15) node {$\B$};
\draw[color=black] (5.4,0) node {$\X$};
\draw[color=black] (7.2,-0.15) node {$\A$};
\draw[color=black] (9.2,-0.15) node {$\B$};
\draw[color=black] (1.3,5.4) node {$\y F$};
\draw[color=black] (2.8,5.4) node {$G$};
\draw[color=black] (4.75,5.4) node {$H_G$};
\draw[color=black] (-3.3,0.9) node {$\y F$};
\draw[color=black] (-1.8,0.9) node {$\y K_i$};
\draw[color=black] (0,0.9) node {$\y H_i$};
\draw[color=black] (6.1,0.9) node {$\y F$};
\draw[color=black] (7.6,0.9) node {$\y K_j$};
\draw[color=black] (9.4,0.9) node {$\y H_j$};
\end{tikzpicture}$$
The question is, where can we get the $K$-s from? As the diagram above suggests, we could get them by a similar representation of $G$ as a cofiltered limit, which is done in Step~\ref{step limiting cone}. This representation goes via an analogous representation of~$G_\ast$. 
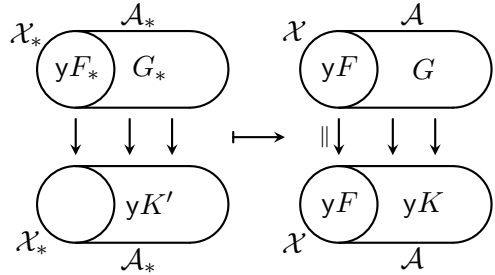
$$\begin{tikzpicture}[line cap=round,line join=round,>=triangle 45,x=1cm,y=1cm,scale=0.5]
\draw [line width=0.8pt] (1.3476,5.1834) circle (1.0164cm);
\draw [line width=0.8pt] (1.3476,6.1998)-- (4.3484,6.1998);
\draw [shift={(4.3484,5.1834)},line width=0.8pt]  plot[domain=-1.5707963267948974:1.570796326794896,variable=\t]({1*1.0164*cos(\t r)+0*1.0164*sin(\t r)},{0*1.0164*cos(\t r)+1*1.0164*sin(\t r)});
\draw [line width=0.8pt] (1.3476,2.6182)-- (4.3484,2.6182);
\draw [shift={(4.3484,1.6018)},line width=0.8pt]  plot[domain=-1.5707963267948974:1.570796326794896,variable=\t]({1*1.0164*cos(\t r)+0*1.0164*sin(\t r)},{0*1.0164*cos(\t r)+1*1.0164*sin(\t r)});
\draw [line width=0.8pt] (1.3476,4.167)-- (4.3484,4.167);
\draw [line width=0.8pt] (1.3476,0.5854)-- (4.3484,0.5854);
\draw [line width=0.8pt] (1.3476,1.6018) circle (1.0164cm);
\draw [line width=0.8pt] (8.2688,5.1834) circle (1.0164cm);
\draw [line width=0.8pt] (8.2688,6.1998)-- (11.2696,6.1998);
\draw [shift={(11.2696,5.1834)},line width=0.8pt]  plot[domain=-1.5707963267948974:1.570796326794896,variable=\t]({1*1.0164*cos(\t r)+0*1.0164*sin(\t r)},{0*1.0164*cos(\t r)+1*1.0164*sin(\t r)});
\draw [line width=0.8pt] (8.2688,4.167)-- (11.2696,4.167);
\draw [line width=0.8pt] (8.2688,1.6018) circle (1.0164cm);
\draw [line width=0.8pt] (8.2688,2.6182)-- (11.2696,2.6182);
\draw [line width=0.8pt] (8.2688,0.5854)-- (11.2696,0.5854);
\draw [shift={(11.2696,1.6018)},line width=0.8pt]  plot[domain=-1.5707963267948974:1.570796326794896,variable=\t]({1*1.0164*cos(\t r)+0*1.0164*sin(\t r)},{0*1.0164*cos(\t r)+1*1.0164*sin(\t r)});
\draw [>=stealth,->,line width=0.8pt] (1.3476,3.865270814814817) -- (1.3476,2.906162074074074);
\draw [>=stealth,->,line width=0.8pt] (2.7668972972972963,3.8652708148148167) -- (2.7668972972972963,2.9061620740740737);
\draw [>=stealth,->,line width=0.8pt] (3.882059459459459,3.8652708148148163) -- (3.882059459459459,2.9061620740740732);
\draw [>=stealth,->,line width=0.8pt] (8.2688,3.865270814814817) -- (8.2688,2.906162074074074);
\draw [>=stealth,->,line width=0.8pt] (9.688097297297297,3.8652708148148167) -- (9.688097297297297,2.9061620740740737);
\draw [>=stealth,->,line width=0.8pt] (10.803259459459458,3.8652708148148163) -- (10.803259459459458,2.9061620740740732);
\draw [>=stealth,->,line width=0.8pt] (5.490530990990992,3.385716444444445) -- (6.851545363825364,3.385716444444445);
\draw [line width=0.8pt] (5.490530990990992,3.5249110962115955)-- (5.490530990990992,3.246521792677295);
\draw[color=black] (0.05,6.1) node {$\X_\ast$};
\draw[color=black] (3,6.6) node {$\A_\ast$};
\draw[color=black] (3.3,1.58) node {$\y K'$};
\draw[color=black] (1.4,5.2) node {$\y F_\ast$};
\draw[color=black] (3.3,5.2) node {$G_\ast$};
\draw[color=black] (3,0.1) node {$\A_\ast$};
\draw[color=black] (0.2,0.58) node {$\X_\ast$};
\draw[color=black] (7.1,6.2) node {$\X$};
\draw[color=black] (10.2,6.65) node {$\A$};
\draw[color=black] (8.3,5.2) node {$\y F$};
\draw[color=black] (10.5,5.2) node {$G$};
\draw[color=black] (7.1,0.58) node {$\X$};
\draw[color=black] (8.3,1.58) node {$\y F$};
\draw[color=black] (10.5,1.58) node {$\y K$};
\draw[color=black] (10.2,0.18) node {$\A$};
\draw[color=black] (7.9,3.39) node {$\mathsmaller{\parallel}$};
\end{tikzpicture}$$
That $G_\ast$ can be presented as a cofiltered limit of $\A_\ast$-structures of the form $\y K'$ for some $K'\in\A_\ast\CC$ is the set-base case of the uniformity lemma from~\cite{DS}, whose proof is incorporated in our Steps~\ref{step cofiltered category}--\ref{step limiting cone}, where we make a direct passage to~$G$ (hence establishing a more general form of the uniformity lemma). What allows this passage is a universal construction carried out in Steps~\ref{first step R}--\ref{last step R}. The steps preceding those are essentially for introducing necessary notation. Finally, in Steps~\ref{step the functor m square}--\ref{step the morphism Hm} the construction of $H_G$ is shown to be functorial, in Steps~\ref{step beta p limiting cone}--\ref{step the isomorphism iG}  we establish an isomorphism $H_G \circ \beta\cong G$ and in the last Step~\ref{step the naturality of iG} we prove its naturality, concluding the proof.

\subsection*{\newstep The pentagon of sketches}

Consider an $\X$-sequent $\cd{\X \ar[r]^-{\alpha} & \A \ar[r]^-{\beta} & \B}$ as in Theorem~\ref{C tilde preserves unconditional exactness properties}. We will need the following subsketches of $\A$ and corresponding subsketch inclusion morphisms between them,
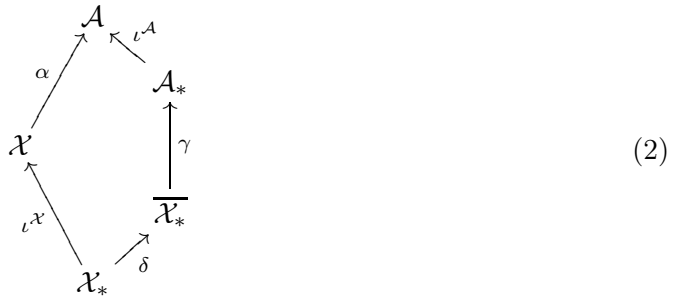
\begin{equation}\label{EquAF}
\vcenter{\cd@=10pt{ & \A & \\ & & \A_\ast\ar[ul]_-{\iota^\A} \\ \X\ar[uur]^-{\alpha} & & \\ & & \overline{\X_\ast}\ar[uu]_-{\gamma} \\ & \X_\ast\ar[uul]^-{\iota^\X}\ar[ur]_-{\delta} & }}
\end{equation}
where $\overline{\X_\ast}$ denotes the extension of $\X_\ast$ within $\A_\ast$ which attaches to $\X_\ast$ all arrows of $\A_\ast$ between objects that lie in $\X_\ast$, and also inherits from $\A_\ast$ all commutativity conditions involving these arrows. The diagram above commutes, as all morphisms are subsketch inclusions. 

\subsection*{\newstep The functor $\y_\A^{\alpha, F}$}

Let $\CC$ be as in Theorem~\ref{C tilde preserves unconditional exactness properties}. Let $F$ be an $\X$-structure in $\CC$. Then $\y F$ is an $\X$-structure in~$\CCC$. Suppose there is a functorial verification of $\alpha\vdash\beta$ for $F$. We will need the following functor induced between the two lateral pullbacks, which are chosen in such a way that the `$\mathsf{J}$' functors are subcategory inclusions:
\begin{equation}\label{EquD}\vcenter{\cd@=30pt{\A^\alpha_F\CC\ar@{}[dddr]|{\textrm{pb}}\ar[ddd]\ar[dr]|-{\mathsf{J}^\alpha_F}\ar@{-->}[rrr]^-{\y_\A^{\alpha, F}} & \ar@{}[dr]|{=} & & \A^\alpha_{\y F}\CCC\ar[ddd]\ar[dl]|-{\mathsf{J}^\alpha_{\y F}}\ar@{}[dddl]|{\textrm{pb}} \\ & \A\CC \ar@{}[rd]|{=}\ar[r]^-{\y_{\A}}
\ar[d]_-{\alpha_\CC} & \A\CCC \ar[d]^-{\alpha_{\CCC}} & \\ & \X\CC \ar[r]_-{\y_{\X}} & \X\CCC  & \\ \1\ar@{=}[rrr]\ar[ur]^-{F} & \ar@{}[ur]|{=} & & \1\ar[ul]_-{\y F}}}
\end{equation}

\subsection*{\newstep Employing Lemma~\ref{LemA}}

In view of Lemma~\ref{LemA}, to prove that there is a functorial verification of $\alpha\vdash\beta$ for $\y F$, it suffices to functorially construct for each $G\in \A^\alpha_{\y F}\CCC$ an $H_G\in\B\CCC$ with an isomorphism $\beta_{\CCC}(H_G)\cong G$ natural in $G$. 

\subsection*{\newstep The functor $L$} 

For this, we first need some preliminary material. In particular, we need to show that for each $G\in \A^\alpha_{\y F}\CCC$, the functor $L\colon(G\downarrow \y_\A^{\alpha,F})\to(G_\ast \downarrow \y_{\A_\ast })$
which arises from the diagram
\begin{equation}\label{EquC}
\vcenter{\cd@=30pt{ (G\downarrow \y_\A^{\alpha,F})\ar@/_10pt/[ddd]\ar[dr]_-{G^{\downarrow \y_\A^{\alpha,F}}}\ar@{-->}[rrrr]^-{L} &  & \ar@{}[d]|{=} &  & (G_\ast\downarrow \y_{\A_\ast})\ar[ld]^-{G_\ast^{\downarrow \y_{\A_\ast}}}\ar@/^10pt/[ddd] \\ 
 & \A^\alpha_F\CC\ar@{}[dr]|{=}\ar[r]^-{\mathsf{J}^\alpha_F}\ar[d]_-{\y_\A^{\alpha,F}} & \A\CC \ar@{}[dr]|{=}\ar[d]|-{\y_{\A}}|(.46){\hole} \ar[r]^-{\iota_\CC^\A} & \A_\ast\CC\ar[d]^-{\y_{\A_\ast}} & \dltwocell<\omit>{<1>\quad\quad\quad\quad\quad(G_\ast,\y_{\A_\ast})} \\ 
 & \A^\alpha_{\y F}\CCC\ar[r]_-{\mathsf{J}^\alpha_{\y F}}\ultwocell<\omit>{<0>(G,\y_\A^{\alpha,F})\;\quad\quad\quad}  & \A\CCC\ar[r]_-{\iota_{\CCC}^\A} & {\A_\ast}\CCC &  \\ \1\ar[ur]^-{G}\ar@{=}[rrrr] & & \ar@{}[u]|{=} & & \1\ar[ul]_-{G_\ast} }}
\end{equation}
by the universal property of the comma category on the right, has a right adjoint with unit of adjunction being an isomorphism. Steps~\ref{first step R} to~\ref{last step R} are devoted to the construction of this right adjoint.

\subsection*{\newstep\label{first step R} The structure $\overline{F_\ast}$}

Since $\y$ is fully faithful, preserves finite limits and colimits and is injective on objects, and since $\delta\colon\X_\ast\to \overline{\X_\ast}$ is bijective on objects, the bottom right square in the diagram
\begin{equation}\label{EquAG}
\vcenter{\cd{\1\ar@{}[ddr]_>>>>>>>>>>>{=}\ar@{}[rrd]^>>>>>>>>>>>{=}\ar@{-->}[dr]|-{\overline{F_\ast}}\ar[dd]_-{F}\ar[r]^-{G} & \A\CCC 
\ar[r]^-{\iota^\A_{\CCC}}  & {\A_\ast}\CCC\ar[d]^-{\gamma_{\CCC}} \\ & \overline{\X_\ast}\CC\ar@{}[dr]|-{\mathrm{pb}}\ar[r]^-{\y_{\overline{\X_\ast}}}\ar[d]_-{\delta_{\CC}} & \overline{\X_\ast}\CCC\ar[d]^-{\delta_{\CCC}} \\ \X\CC \ar[r]_-{\iota^\X_\CC} & \X_\ast\CC\ar[r]_-{\y_{\X_\ast}} & \X_\ast\CCC}}
\end{equation}
is a pullback. Since the diagram of solid arrows commutes, we get the dashed arrow keeping the diagram commute. 

\subsection*{\newstep The natural transformation $\mathfrak{f}$}

Since $\y_{\overline{\X_\ast}}$ is full and faithful, we get a unique natural transformation $\mathfrak{f}$ in the left square below fulfilling the identity
\begin{equation}\label{EquB}
\vcenter{\cd@=20pt{(G_\ast\downarrow\y_{\A_\ast})\ar[r]^-{G_\ast^{\downarrow \y_{\A_\ast}}}\ar[d] & \A_\ast\CC\ar@{}[rdd]^-{=}\ar[d]^-{\gamma_\CC}\ar[r]^-{\y_{\A_\ast}} & {\A_\ast}\CCC\ar[dd]^-{\gamma_{\CCC}} \\ \1\ar@{}[rrd]_-{=}\ar[d]_-{G_\ast}\ar[r]_-{\overline{F_\ast}} & \overline{\X_\ast}\CC\ar[dr]_-{\y_{\overline{\X_\ast}}}\ultwocell<\omit>{<0>\mathfrak{f}} & \\ {\A_\ast}\CCC\ar[rr]_-{\gamma_{\CCC}} &  & \overline{\X_\ast}\CCC}} 
\quad\equiv
\vcenter{\cd@=40pt{(G_\ast\downarrow\y_{\A_\ast}) \ar[r]^-{G_\ast^{\downarrow \y_{\A_\ast}}}\ar[d] & \A_\ast\CC\ar[r]^-{\y_{\A_\ast}}\ar[d]^-{\y_{\A_\ast}}\ar@{}[rdd]^-{=} & {\A_\ast}\CCC\ar[dd]^-{\gamma_{\CCC}} \\ \1\ar[r]_-{G_\ast}\ar[d]_-{G_\ast}\ar@{}[rrd]_-{=} & {\A_\ast}\CCC\ar[dr]_-{\gamma_{\CCC}}\ultwocell<\omit>{<1.5>(G_\ast,\y_{\A_\ast})\quad\quad\quad} & \\ {\A_\ast}\CCC\ar[rr]_-{\gamma_{\CCC}} & & \overline{\X_\ast}\CCC.}}\end{equation}

\subsection*{\newstep The functor $\gamma_\CC'$ and the natural transformation $\mathfrak{u}$} 

Since $\A_\ast$ is the underlying sketch of a finite category further equipped with the commutativity conditions from $\overline{\X_\ast}$ and $\overline{\X_\ast}$ is a full and regular subsketch of $\A_\ast$, the functor $\gamma_\CC$ has a right-adjoint-right-inverse $\gamma_\CC'$, with counit of adjunction being an identity natural transformation (see Subsection~\ref{subsection kan extensions}). Let $\mathfrak{u}\colon 1_{\A_\ast\CC} \Rightarrow \gamma_\CC'\circ\gamma_\CC$ denote the unit of this adjunction. Then, the triangular identities give us
\begin{equation}
\label{EquAA} \gamma_\CC\bullet\mathfrak{u}=1_{\gamma_\CC}
\end{equation}
and
\begin{equation}
\label{EquO} \mathfrak{u}\bullet\gamma'_\CC=1_{\gamma'_\CC}.
\end{equation}

\subsection*{\newstep The functor $\gamma'_{\CCC}$ and the natural transformation $\mathfrak{v}$}

Moreover, as explained in Subsection~\ref{subsection kan extensions}, we can construct a right-adjoint-right-inverse $\gamma'_{\CCC}$ of~$\gamma_{\CCC}$, with counit of adjunction being an identity and unit denoted by $\mathfrak{v}\colon 1_{\A_\ast\CCC} \Rightarrow \gamma'_{\CCC}\circ\gamma_{\CCC}$ and such that the diagram 
\begin{equation}\label{EquAE}
\vcenter{\cd{\A_\ast\CC\ar[r]^-{\y_{\A_\ast}} & {\A_\ast}\CCC \\ \overline{\X_\ast}\CC\ar[r]_-{\y_{\overline{\X_\ast}}}\ar[u]^-{\gamma'_\CC} & \overline{\X_\ast}\CCC\ar[u]_-{\gamma'_{\CCC}}}}
\end{equation}
commutes and
\begin{equation}\label{EquAB}
\y_{\A_\ast}\bullet\mathfrak{u}=\mathfrak{v}\bullet\y_{\A_\ast}.
\end{equation}
The triangular identities give
\begin{equation}\label{EquAC}
\gamma_{\CCC}\bullet\mathfrak{v}=1_{\gamma_{\CCC}}
\end{equation}
and
\begin{equation*}
\mathfrak{v}\bullet\gamma'_{\CCC}=1_{\gamma'_{\CCC}}.
\end{equation*}

\subsection*{\newstep The functor $S$ and the natural transformations $\mathfrak{s}_1$ and $\mathfrak{s}_2$} 

Since $\A_\ast\CC$ has finite limits, we can consider the following pullback of natural transformations:
\begin{equation}\label{EquA}\vcenter{\cd@!=30pt{ & S\ar@{=>}[dl]_-{\mathfrak{s}_1}\ar@{=>}[dr]^-{\mathfrak{s}_2}\ar@{}[dd]|{\textrm{pb}} & \\ \gamma'_\CC\circ\overline{F_\ast}\circ\mathrel{!}_{(G_\ast\downarrow\y_{\A_\ast})}\ar@{=>}[rd]_-{\gamma'_\CC\bullet\mathfrak{f}} & & G_\ast^{\downarrow \y_{\A_\ast}}\ar@{=>}[dl]^-{\mathfrak{u}\bullet G_\ast^{\downarrow \y_{\A_\ast}}} \\ & \gamma'_\CC\circ \gamma_\CC\circ G_\ast^{\downarrow \y_{\A_\ast}} & }}\end{equation}
Since $\gamma_\CC$ preserves finite limits and in view of~(\ref{EquAA}), composing the above pullback with $\gamma_\CC$ results in the pullback
$$\cd@!=30pt{ & \gamma_\CC\circ S\ar@{=>}[dl]_-{\gamma_\CC\bullet\mathfrak{s}_1}\ar@{=>}[dr]^-{\gamma_\CC\bullet\mathfrak{s}_2}\ar@{}[dd]|{\textrm{pb}} & \\ \overline{F_\ast}\circ\mathrel{!}_{(G_\ast\downarrow\y_{\A_\ast})}\ar@{=>}[rd]_-{\mathfrak{f}} & & \gamma_\CC\circ G_\ast^{\downarrow \y_{\A_\ast}}\ar@{=}[dl] \\ & \gamma_\CC\circ G_\ast^{\downarrow \y_{\A_\ast}} & }$$ 
and hence $\gamma_\CC\bullet\mathfrak{s}_1$ is a natural isomorphism. Using Lemma~\ref{LemB} for the isomorphism $\gamma_\CC\bullet\mathfrak{s}_1$ evaluated at any object of $(G_\ast\downarrow\y_{\A_\ast})$, we could redefine pullback~(\ref{EquA}) so that $\gamma_\CC\bullet\mathfrak{s}_1$ is an identity natural transformation, i.e.,
\begin{equation}\label{EquN}
\gamma_\CC\bullet\mathfrak{s}_1=1_{\gamma_\CC\circ S}.
\end{equation}
Thus, without loss of generality, we assume $\gamma_\CC\bullet\mathfrak{s}_2=\mathfrak{f}$.

\subsection*{\newstep The natural transformation $\mathfrak{s}$}

The fact that $\y$ preserves finite limits also gives that composing the pullback~(\ref{EquA}) with $\y_{\A_\ast}$ results in a pullback. In view of~(\ref{EquB}), (\ref{EquAE}) and~(\ref{EquAB}), this pullback can be computed to be the bottom square in the following diagram:
\begin{equation}\label{EquE}\vcenter{\cd@!=50pt{ & G_\ast\circ\mathrel{!}_{(G_\ast\downarrow\y_{\A_\ast})}\ar@{==>}[d]^-{\mathfrak{s}}\ar@/_30pt/@{=>}[ddl]_-{\mathfrak{v}\bullet (G_\ast\circ\mathrel{!}_{(G_\ast\downarrow\y_{\A_\ast})})}\ar@/^30pt/@{=>}[ddr]^-{(G_\ast,\y_{\A_\ast})} & \\ & \y_{\A_\ast}\circ S\ar@{=>}[dl]_-{\y_{\A_\ast}\bullet\mathfrak{s}_1}\ar@{=>}[dr]^-{\y_{\A_\ast}\bullet\mathfrak{s}_2}\ar@{}[dd]|{\textrm{pb}} & \\ \gamma'_{\CCC}\circ\gamma_{\CCC}\circ G_\ast\circ\mathrel{!}_{(G_\ast\downarrow\y_{\A_\ast})}\ar@{=>}[rd]_-{(\gamma'_{\CCC}\circ \gamma_{\CCC})\bullet(G_\ast,\y_{\A_\ast})} & & \y_{\A_\ast}\circ G_\ast^{\downarrow \y_{\A_\ast}}\ar@{=>}[dl]^-{\mathfrak{v}\bullet(\y_{\A_\ast}\circ G_\ast^{\downarrow \y_{\A_\ast}})} \\ & \gamma'_{\CCC}\circ \gamma_{\CCC}\circ \y_{\A_\ast}\circ G_\ast^{\downarrow \y_{\A_\ast}} & }}\end{equation}
Since the outer diagram commutes, we obtain an induced natural transformation $\mathfrak{s}$ making the two triangular diagrams above commute.
Since $\gamma_{\CCC} \bullet \mathfrak{v}=1_{\gamma_{\CCC}}$~(\ref{EquAC}) and $(\gamma_{\CCC}\circ\y_{\A_\ast})\bullet \mathfrak{s}_1=1_{\y_{\overline{\X_\ast}}\circ\gamma_{\CC}\circ S}$ (by~(\ref{EquN})), we have in particular that
\begin{equation}\label{EquU}\gamma_{\CCC} \bullet \mathfrak{s}=1_{\y_{\overline{\X_\ast}}\circ\gamma_{\CC}\circ S}.\end{equation}

\subsection*{\newstep The functor $T$}

On one hand, the top rectangle in the diagram \vspace{5pt}
\begin{equation}\label{EquI}\vcenter{\cd{ (G_\ast\downarrow \y_{\A_\ast})\ar@{-->}[r]^-{T}\ar@/_10pt/[rdd]_-{S}\ar@/^20pt/@<2pt>[rrr] & \A^\alpha_F\CC\ar[d]_-{\mathsf{J}^\alpha_F}\ar[rr]\ar@{}[rrd]|{\textrm{pb}} & & \1\ar[d]^-{F} \\ & \A\CC\ar@{}[rrd]|{\textrm{pb}}\ar[d]_-{\iota_\CC^\A}\ar[rr]^-{\alpha_\CC} & & \X\CC\ar[d]^-{\iota^\X_\CC} \\ \ar@{}[uur]|>>>>>>>>>>>{=} & \A_\ast\CC\ar[r]_-{\gamma_\CC} & \overline{\X_\ast}\CC\ar[r]_-{\delta_\CC} & \X_\ast\CC }}\end{equation}
is a pullback by definition of $\A^\alpha_F\CC$. Moreover, since $\alpha$ is unconditional of finite kind, the convergence conditions of $\A$ already lie in~$\X$. Therefore, each $\A_\ast$-structure in $\CC$, whose restriction as an $\X_\ast$-structure extends as an $\X$-structure, can be extended to an $\A$-structure, showing that the bottom rectangle is also a pullback. On the other other hand, by~(\ref{EquAG}) and~(\ref{EquN}), the outer diagram commutes. This gives rise to the functor $T$ which makes the left triangular diagram commute.

\subsection*{\newstep The natural transformation $\mathfrak{t}$}

Since the two rectangles in the diagram
$$\cd{(G_\ast\downarrow \y_{\A_\ast}) \ar@{}@<1pt>[rd]|(.25){}="I" \ar@<3pt>"I";[rd]^(.4){\y_\A^{\alpha,F} \circ T} \ar@{}@<0pt>[rd]|(.13){}="J" \ar@<-3pt>"J";[rd]_(.55){G \circ \mathrel{!}_{(G_\ast\downarrow\y_{\A_\ast})}} \ar@{}@<-3pt>[r]|(.1){}="E" \ar@{}@<2pt>[d]|(-.12){}="G" \ar@{}@<2pt>[d]|(.1){}="K" \ar@{}@<34pt>[rrrrrd]|-{}="B" \ar@/^2pc/@<-4pt>[rrrrrd]_-{G_\ast\circ\mathrel{!}_{(G_\ast\downarrow\y_{\A_\ast})}} \ar@{}@<23pt>[rrrrrd]|-{}="A" \ar@{=>}"A";"B"_(.4){\,\mathfrak{s}} \ar@{}@<-34pt>[rddd]|-{}="D" \ar@{}@<-43pt>[rddd]|-{}="C" \ar@{=>}"C";"D" &&& \\ & \A^\alpha_{\y F}\CCC \ar[rr]_-{\mathsf{J}^\alpha_{\y F}} \ar[dd] \ar@{}[rrdd]|{\textrm{pb}} && \A\CCC \ar[rr]_-{\iota_{\CCC}^\A} \ar[dd]_-{\alpha_{\CCC}} \ar@{}[rrdd]|{\textrm{pb}} && \A_{\ast}\CCC \ar[d]^-{\gamma_{\CCC}} \ar@{}[l]|(.1){}="F" \ar@/^2pc/@<14pt>"E";"F"^-{\y_{\A_{\ast}} \circ S} \\ &&&&& \overline{\X_\ast}\CCC \ar[d]^-{\delta_{\CCC}} \\ & \1 \ar[rr]_-{\y F} \ar@{}@<-15pt>[u]|(.23){}="H" \ar@{}@<-1pt>[u]|(.24){}="L" \ar@/_2.5pc/@<-26pt>"G";"H" \ar@/_2.2pc/@<-8pt>"K";"L" && \X\CCC \ar[rr]_-{\iota_{\CCC}^\X} && \X_{\ast}\CCC}$$
are strong pullbacks, in view of~(\ref{EquU}), there is a unique natural transformation $\mathfrak{t}$ such that
\begin{equation}\label{EquF}\vcenter{\cd@=28pt{ (G_\ast\downarrow \y_{\A_\ast})\ar[r]^-{T}\ar[d] & \A^\alpha_F\CC\ar@{}[rd]|{=}\ar[r]^-{\mathsf{J}^\alpha_F}\ar[d]_-{\y_\A^{\alpha,F}} & \A\CC\ar@{}[rd]|{=}\ar[r]^-{\iota^\A_\CC}\ar[d]^-{\y_\A} & \A_\ast\CC\ar[d]^-{\y_{\A_\ast}} \\ \1\ar[r]_-{G} & \A^\alpha_{\y F}\CCC\ultwocell<\omit>{<0>\mathfrak{t}}\ar[r]_-{\mathsf{J}^\alpha_{\y F}} & \A\CCC\ar[r]_-{\iota^\A_{\CCC}} & {\A_\ast}\CCC}}\equiv\vcenter{\cd@=20pt{(G_\ast\downarrow \y_{\A_\ast})\ar[r]^-{S}\ar[d] & \A_\ast\CC\ar[d]^-{\y_{\A_\ast}} \\ \1\ar[r]_-{G_\ast} & {\A_\ast}\CCC.\ultwocell<\omit>{<0>\mathfrak{s}}}}\end{equation}

\subsection*{\newstep The functor $R$}

Now, by the universal property of the comma construction, we have a functor $R$ as displayed here:
\begin{equation}\label{EquG}\vcenter{\cd@=30pt{ & (G\downarrow \y_\A^{\alpha,F})\ar[d]^-{\!G^{\downarrow \y_\A^{\alpha,F}}}\ar[r]\drtwocell<\omit>{<1>\,\,\,\,\quad\quad(G,\y_\A^{\alpha,F})} & \1\ar[d]^-{G}\\ (G_\ast\downarrow \y_{\A_\ast})\ar@{}[urr]|<<<<<<<<<{=}\ar[r]_-{T}\ar[ur]^-{R} & \A^\alpha_F\CC\ar[r]_-{\y_{\A}^{\alpha,F}} & \A^\alpha_{\y F}\CCC }}\equiv\vcenter{\cd@=30pt{(G_\ast\downarrow \y_{\A_\ast})\ar[d]_-{T}\ar[r]\drtwocell<\omit>{<0> \mathfrak{t}} & \1\ar[d]^-{G}\\ \A^\alpha_F\CC\ar[r]_-{\y_{\A}^{\alpha,F}} & \A^\alpha_{\y F}\CCC }}\end{equation}
We are now to show that this functor is a right adjoint of $L$ (few more steps will be needed for this). We will proceed as follows. First we are going to build candidates for the unit and the counit of adjunction, and then we are going to show that they satisfy the triangular identities required for an adjunction.

\subsection*{\newstep $\mathfrak{f}\bullet L$ is an identity}

To build the candidate for the unit, we first go back to the pullback~(\ref{EquA}) and compose it with $L$ to get the following pullback.
$$\vcenter{\cd@!=30pt{ & S\circ L\ar@{=>}[dl]_-{\mathfrak{s}_1\bullet L}\ar@{=>}[dr]^-{\mathfrak{s}_2\bullet L}\ar@{}[dd]|{\textrm{pb}} & \\ \gamma'_\CC\circ\overline{F_\ast}\circ\mathrel{!}_{(G\downarrow\y_{\A}^{\alpha,F})}\ar@{=>}[rd]_-{\gamma'_\CC\bullet\mathfrak{f}\bullet L} & & G_\ast^{\downarrow \y_{\A_\ast}}\circ L\ar@{=>}[dl]^-{\mathfrak{u}\bullet (G_\ast^{\downarrow \y_{\A_\ast}}\circ L)} \\ & \gamma'_\CC\circ \gamma_\CC\circ G_\ast^{\downarrow \y_{\A_\ast}}\!\!\circ L & }}$$
We will now show that $\mathfrak{f}\bullet L$ is an identity natural transformation, from which we will be able to conclude that $\gamma'_\CC\bullet\mathfrak{f}\bullet L$ is an identity and therefore $\mathfrak{s}_2\bullet L$ is an isomorphism. Since $\y_{\overline{\X_\ast}}$ is faithful, in view of~(\ref{EquB}), it suffices to show that $\gamma_{\CCC}\bullet(G_\ast,\y_{A_\ast})\bullet L$ is an identity. In view of~(\ref{EquAF}) and~(\ref{EquC}) and since $\delta_{\CCC}$ is faithful, this reduces to showing that $$(\delta_{\CCC}\circ\gamma_{\CCC})\bullet (G_{\ast},\y_{\A_\ast}) \bullet L=(\iota^{\X}_{\CCC}\circ \alpha_{\CCC}\circ \mathsf{J}^\alpha_{\y F})\bullet(G,\y_\A^{\alpha,F})$$  
is an identity. Indeed, in view of~(\ref{EquD}),
\begin{align*}
(\iota^{\X}_{\CCC}\circ \alpha_{\CCC}\circ \mathsf{J}^\alpha_{\y F})\bullet(G,\y_\A^{\alpha,F}) & =(\iota^{\X}_{\CCC}\circ (\y F)\circ \mathrel{!}_{\A^{\alpha}_{\y F}\CCC}) \bullet(G,\y_\A^{\alpha,F}) \\ & =(\iota^{\X}_{\CCC}\circ (\y F))\bullet 1_{\mathrel{!}_{(G\downarrow \y_\A^{\alpha,F})}}\\ &=1_{\iota^{\X}_{\CCC}\circ (\y F)\circ \mathrel{!}_{(G\downarrow \y_\A^{\alpha,F})}}. 
\end{align*} 
So the above pullback becomes the following one.
\begin{equation}\label{EquH}\vcenter{\cd@!=30pt{ & S\circ L\ar@{=>}[dl]_-{\mathfrak{s}_1\bullet L}\ar@{=>}[dr]^-{\mathfrak{s}_2\bullet L}\ar@{}[dd]|{\textrm{pb}} & \\ \gamma'_\CC\circ\overline{F_\ast}\circ\mathrel{!}_{(G\downarrow\y_{\A}^{\alpha,F})}\ar@{=}[rd]
&& G_\ast^{\downarrow \y_{\A_\ast}}\circ L\ar@{=>}[dl]^-{\mathfrak{u}\bullet (G_\ast^{\downarrow \y_{\A_\ast}}\circ L)} \\ & \gamma'_\CC\circ \gamma_\CC\circ G_\ast^{\downarrow \y_{\A_\ast}}\!\!\circ L & }}\end{equation}

\subsection*{\newstep The natural isomorphism $\mathfrak{w}$}

Since $(\delta_\CC \circ \gamma_\CC) \bullet \mathfrak{s}_2 \bullet L = \delta_\CC \bullet \mathfrak{f} \bullet L$ is an identity natural transformation and the two rectangles in the diagram
$$\cd{(G\downarrow \y_\A^{\alpha,F}) \ar@{}@<1pt>[rd]|(.28){}="I" \ar@<3pt>"I";[rd] \ar@{}@<6pt>[rd]^(.6){G^{\downarrow \y_\A^{\alpha,F}}\circ R \circ L} \ar@{}[rd]|(.17){}="J" \ar@<-3pt>"J";[rd]_-{G^{\downarrow \y_\A^{\alpha,F}}} \ar@{}@<40pt>[rrrrrd]|-{}="B" \ar@{}@<-1pt>[r]|(.1){}="E" \ar@{}[d]|(-.1){}="G" \ar@{}@<2pt>[d]|(.15){}="K" \ar@/^2pc/@<-3pt>[rrrrrd]_-{G_\ast^{\downarrow \y_{\A_\ast}} \circ L} \ar@{}@<24pt>[rrrrrd]|-{}="A" \ar@{=>}"A";"B"_(.28){(\mathfrak{s}_2 \bullet L)^{-1}} \ar@{}@<-29pt>[rddd]|-{}="D" \ar@{}@<-39pt>[rddd]|-{}="C" \ar@{=>}"C";"D" &&& \\ & \A^\alpha_F\CC \ar[rr]_-{\mathsf{J}^\alpha_F} \ar[dd] \ar@{}[rrdd]|{\textrm{pb}} && \A\CC \ar[rr]_-{\iota_{\CC}^\A} \ar[dd]_-{\alpha_{\CC}} \ar@{}[rrdd]|{\textrm{pb}} && \A_{\ast}\CC \ar[dd]^-{\delta_{\CC} \circ \gamma_{\CC}} \ar@{}[l]|(.1){}="F" \ar@/^2.5pc/@<14pt>"E";"F"^-{S \circ L} \\ &&& \\ & \1 \ar[rr]_-{F} \ar@{}@<-13pt>[u]|(.26){}="H" \ar@{}@<-1pt>[u]|(.26){}="L" \ar@/_2.2pc/@<-24pt>"G";"H" \ar@/_1.7pc/@<-8pt>"K";"L" && \X\CC \ar[rr]_-{\iota_{\CC}^\X} && \X_{\ast}\CC}$$
are strong pullbacks, there exists a unique natural isomorphism $\mathfrak{w}$ such that
\begin{equation}\label{EquJ}
\vcenter{\cd{(G\downarrow\y_\A^{\alpha,F})\ar@{=}[dd]\ar[r]^-{L} & (G_\ast\downarrow\y_{\A_\ast})\ar[d]^-{R}\ar@{=}[r]\ar@{}[rdd]|-{=} & (G_\ast\downarrow\y_{\A_\ast})\ar[ddd]^-{S} \\ & (G\downarrow\y_\A^{\alpha,F})\ar[d]^-{G^{\downarrow \y_\A^{\alpha,F}}} & \\ (G\downarrow\y_\A^{\alpha,F})\ar[d]_-{L}\ar[r]_-{G^{\downarrow \y_\A^{\alpha,F}}}  & \A^\alpha_F\CC\uultwocell<\omit>{<0> \mathfrak{w}}\ar[rd]|(.45){\hole}|-{\iota_\CC^\A\circ\mathsf{J}^\alpha_F} & \\ (G_\ast\downarrow\y_{\A_\ast})\ar@{}[ur]_-{=}\ar[rr]_-{G_\ast^{\downarrow \y_{\A_\ast}}} & & \A_\ast\CC}}\equiv\vcenter{\cd@=30pt{(G\downarrow\y_\A^{\alpha,F})\ar[r]^-{L}\ar[d]_-{L} & (G_\ast\downarrow\y_{\A_\ast})\ar[d]^-{S} \\ (G_\ast\downarrow\y_{\A_\ast})\ar[r]_-{G_\ast^{\downarrow\y_{\A_\ast}}} & \A_\ast\CC. \ultwocell<\omit>{<0> (\mathfrak{s}_2\bullet L)^{-1}\quad\quad\; }}}\end{equation}

\subsection*{\newstep The natural unit $\mathfrak{y}$}

We want to lift $\mathfrak{w}$ to a natural isomorphism $\mathfrak{y}\colon 1_{(G\downarrow \y_\A^{\alpha,F})}\Rightarrow R\circ L$, which will be the required candidate for the unit of the adjunction. We will do this by using the universal property of the comma construction
$$\cd@=40pt{(G\downarrow \y_\A^{\alpha,F})\ar[d]_-{G^{\downarrow \y_\A^{\alpha,F}}}\ar[r]\drtwocell<\omit>{<1>\,\,\,\,\quad\quad(G,\y_\A^{\alpha,F})} & \1\ar[d]^-{G}\\ \A^\alpha_F\CC\ar[r]_-{\y_\A^{\alpha,F}} & \A^\alpha_{\y F}\CCC.}$$
To be able to apply the universal property to produce such lift of $\mathfrak{w}$, we must show
$$(\y_\A^{\alpha,F}\bullet\mathfrak{w})\circ(G,\y_\A^{\alpha,F})=(G,\y_\A^{\alpha,F})\bullet (R\circ L).$$ Since $\iota_{\CCC}^\A\circ\mathsf{J}^\alpha_{\y F}$ is faithful, it is sufficient to show that the two expressions of the equality above are equal after being composed with this functor. By the interchange law, 
$$(\iota_{\CCC}^\A\circ\mathsf{J}^\alpha_{\y F}) \bullet ((\y_\A^{\alpha,F}\bullet\mathfrak{w})\circ(G,\y_\A^{\alpha,F}))=((\iota_{\CCC}^\A\circ\mathsf{J}^\alpha_{\y F} \circ \y_\A^{\alpha,F})\bullet\mathfrak{w})\circ ((\iota_{\CCC}^\A\circ\mathsf{J}^\alpha_{\y F})\bullet(G,\y_\A^{\alpha,F})).$$
By the top trapezium in~(\ref{EquD}), the middle-right square in~(\ref{EquC}), and in view of~(\ref{EquJ}), we have:
$$(\iota_{\CCC}^\A\circ\mathsf{J}^\alpha_{\y F} \circ \y_\A^{\alpha,F})\bullet\mathfrak{w}=\y_{\A_\ast}\bullet(\mathfrak{s}_2\bullet L)^{-1}.$$
By commutativity of~(\ref{EquC}) and the top right triangular diagram in~(\ref{EquE}), we have:
$$(\iota_{\CCC}^\A\circ\mathsf{J}^\alpha_{\y F})\bullet(G,\y_\A^{\alpha,F})=((\y_{\A_\ast}\bullet\mathfrak{s}_2)\circ \mathfrak{s})\bullet L=(\y_{\A_\ast}\bullet\mathfrak{s}_2\bullet L)\circ (\mathfrak{s}\bullet L).$$
Putting it all together and applying~(\ref{EquF}) and~(\ref{EquG}), we get
\begin{align*}
(\iota_{\CCC}^\A\circ\mathsf{J}^\alpha_{\y F}) \bullet ((\y_\A^{\alpha,F}\bullet\mathfrak{w})\circ(G,\y_\A^{\alpha,F})) &=(\y_{\A_\ast}\bullet(\mathfrak{s}_2\bullet L)^{-1})\circ(\y_{\A_\ast}\bullet\mathfrak{s}_2\bullet L)\circ (\mathfrak{s}\bullet L) \\ &= s\bullet L \\ &= (\iota_{\CCC}^{\A}\circ \mathsf{J}^\alpha_{\y F})\bullet\mathfrak{t}\bullet L\\ &=(\iota_{\CCC}^{\A}\circ \mathsf{J}^\alpha_{\y F})\bullet(G,\y_{\A}^{\alpha,F})\bullet (R\circ L),
\end{align*}
as desired. So there exists a unique natural isomorphism $\mathfrak{y}\colon 1_{(G\downarrow \y_\A^{\alpha,F})}\Rightarrow R\circ L$ such that
\begin{equation}\label{EquK} G^{\downarrow\y_\A^{\alpha,F}}\bullet\mathfrak{y}=\mathfrak{w}.\end{equation}

\subsection*{\newstep The natural counit $\mathfrak{e}$}

The next step is to construct a candidate for the counit of adjunction $L\dashv R$. We will do this using the universal property of the comma construction
$$\cd@=50pt{(G_\ast\downarrow \y_{\A_\ast})\ar[d]_-{G_\ast^{\downarrow \y_{\A_\ast}}}\ar[r]\drtwocell<\omit>{<1>\quad\quad\quad(G_\ast,\y_{\A_\ast})} & \1\ar[d]^-{G_\ast}\\ \A_\ast\CC\ar[r]_-{\y_{\A_\ast}} & {\A_\ast}\CCC.}$$  
Notice that by commutativity the top trapezium in~(\ref{EquC}), and by~(\ref{EquG}) and~(\ref{EquI}), we have 
\begin{equation}\label{EquM}
G_\ast^{\downarrow \y_{\A_\ast}}\!\!\circ L\circ R=S.
\end{equation} 
Consider then the natural transformation $\mathfrak{s}_2\colon G_\ast^{\downarrow \y_{\A_\ast}}\!\!\circ L\circ R\Rightarrow G_\ast^{\downarrow \y_{\A_\ast}}$. To lift this to a natural transformation $L\circ R\Rightarrow 1_{(G_\ast\downarrow \y_{\A_\ast})}$, we must show $$(\y_{\A_\ast}\bullet\mathfrak{s}_2)\circ((G_\ast,\y_{\A_\ast})\bullet(L\circ R))=(G_\ast, \y_{\A_\ast}).$$
This equality follows directly by first applying commutativity of~(\ref{EquC}), and then~(\ref{EquG}) followed by~(\ref{EquF}), and finally, applying commutativity of the right triangular diagram in~(\ref{EquE}). We thus obtain a unique natural transformation $\mathfrak{e}\colon L\circ R\Rightarrow 1_{(G_\ast\downarrow \y_{\A_\ast})}$ such that
\begin{equation}
\label{EquL} G_\ast^{\downarrow\y_{\A_\ast}}\bullet\mathfrak{e}=\mathfrak{s}_2.
\end{equation}

\subsection*{\newstep First triangular identity}

We now begin proving the triangular identities. To prove $(\mathfrak{e}\bullet L)\circ (L\bullet \mathfrak{y})=1_L$, it is sufficient, by faithfulness of $G_\ast^{\downarrow\y_{\A_\ast}}$, to prove 
$$G_\ast^{\downarrow\y_{\A_\ast}}\bullet((\mathfrak{e}\bullet L)\circ (L\bullet \mathfrak{y}))=1_{G_\ast^{\downarrow\y_{\A_\ast}}\circ L}.$$
Indeed, we have:
\begin{align*}
G_\ast^{\downarrow\y_{\A_\ast}}\bullet((\mathfrak{e}\bullet L)\circ (L\bullet \mathfrak{y})) & = (G_\ast^{\downarrow\y_{\A_\ast}}\bullet\mathfrak{e}\bullet L)\circ (G_\ast^{\downarrow\y_{\A_\ast}}\bullet L\bullet \mathfrak{y}) & \textrm{[middle interchange]}\\ &= (\mathfrak{s}_2\bullet L)\circ (G_\ast^{\downarrow\y_{\A_\ast}}\bullet L\bullet \mathfrak{y}) & \textrm{[by (\ref{EquL})]}
\\ &= (\mathfrak{s}_2\bullet L)\circ (\mathfrak{s}_2\bullet L)^{-1} & [\textrm{by (\ref{EquC}), (\ref{EquK}), (\ref{EquJ})}]\\
&=1_{G_\ast^{\downarrow\y_{\A_\ast}}\circ L}.
\end{align*}

\subsection*{\newstep\label{last step R} Second triangular identity}

Now, since $\mathfrak{y}$ is an isomorphism, and the first triangular identity is already established, to get the second triangular identity $(R\bullet\mathfrak{e})\circ(\mathfrak{y}\bullet R)=1_{R}$ it is sufficient to prove $(L\circ R)\bullet\mathfrak{e}=\mathfrak{e}\bullet(L\circ R)$. Since $G_\ast^{\downarrow\y_{\A_\ast}}$ is faithful and $\mathfrak{s}_1$, $\mathfrak{s}_2$ are pullback projections and hence are jointly monomorphic, it is further sufficient to establish the following two identities:
\begin{align*}
\mathfrak{s}_1\circ((G_\ast^{\downarrow\y_{\A_\ast}}\circ L\circ R)\bullet\mathfrak{e})=\mathfrak{s}_1\circ(G_\ast^{\downarrow\y_{\A_\ast}}\bullet\mathfrak{e}\bullet(L\circ R)),\\
\mathfrak{s}_2\circ ((G_\ast^{\downarrow\y_{\A_\ast}}\circ L\circ R)\bullet\mathfrak{e})=\mathfrak{s}_2\circ (G_\ast^{\downarrow\y_{\A_\ast}}\bullet\mathfrak{e}\bullet(L\circ R)).
\end{align*}
We begin with the first one:
\begin{align*}
\mathfrak{s}_1\circ((G_\ast^{\downarrow\y_{\A_\ast}}\circ L\circ R)\bullet\mathfrak{e}) &= \mathfrak{s}_1\circ(S\bullet\mathfrak{e}) & \textrm{[by (\ref{EquM})]}\\
&=\mathfrak{s_1}\bullet\mathfrak{e}\\
&=((\gamma'_\CC\circ\overline{F_\ast}\circ\mathrel{!}_{(G_\ast\downarrow\y_{\A_\ast})})\bullet\mathfrak{e})\circ(\mathfrak{s}_1\bullet (L\circ R)) \\
&=\mathfrak{s}_1\bullet (L\circ R)=(\mathfrak{s}_1\bullet L)\bullet R \\
&=(\mathfrak{u}\bullet(G_\ast^{\downarrow\y_{\A_\ast}}\circ L\circ R))\circ(s_2\bullet(L\circ R)) & \textrm{[by (\ref{EquH})]} \\
&=(\mathfrak{u}\bullet S)\circ(\mathfrak{s}_2\bullet(L\circ R)) & \textrm{[by (\ref{EquM})]} \\
&=(\gamma'_\CC\bullet 1_{\gamma_\CC\circ S})\circ(\mathfrak{u}\bullet S)\circ(\mathfrak{s}_2\bullet(L\circ R))  \\
&=((\gamma'_\CC\circ \gamma_\CC)\bullet\mathfrak{s}_1)\circ(\mathfrak{u}\bullet S)\circ(\mathfrak{s}_2\bullet(L\circ R)) & \textrm{[by (\ref{EquN})]}  \\
&=(\mathfrak{u}\bullet \mathfrak{s}_1)\circ (\mathfrak{s}_2\bullet(L\circ R)) & \\
&=(\mathfrak{u}\bullet(\gamma'_\CC\circ\overline{F_\ast}\circ\mathrel{!}_{(G_\ast\downarrow\y_{\A_\ast})}))\circ\mathfrak{s}_1\circ (\mathfrak{s}_2\bullet(L\circ R)) \\
&=((\mathfrak{u}\bullet\gamma'_\CC)\bullet(\overline{F_\ast}\circ\mathrel{!}_{(G_\ast\downarrow\y_{\A_\ast})}))\circ\mathfrak{s}_1\circ (\mathfrak{s}_2\bullet(L\circ R)) \\
&=\mathfrak{s}_1\circ (\mathfrak{s}_2\bullet(L\circ R)) & \textrm{[by (\ref{EquO})]}\\
&=\mathfrak{s}_1\circ (G_\ast^{\downarrow\y_{\A_\ast}}\bullet\mathfrak{e}\bullet(L\circ R)) & \textrm{[by (\ref{EquL})]}
\end{align*}
The second identity is easier to obtain:
\begin{align*}
\mathfrak{s}_2\circ ((G_\ast^{\downarrow\y_{\A_\ast}}\circ L\circ R)\bullet\mathfrak{e}) & = (G_\ast^{\downarrow\y_{\A_\ast}}\bullet\mathfrak{e})\circ((G_\ast^{\downarrow\y_{\A_\ast}}\circ L\circ R)\bullet\mathfrak{e}) & \textrm{[by (\ref{EquL})]}\\
&=G_\ast^{\downarrow\y_{\A_\ast}}\bullet\mathfrak{e}\bullet\mathfrak{e}\\
&=(G_\ast^{\downarrow\y_{\A_\ast}}\bullet\mathfrak{e})\circ(G_\ast^{\downarrow\y_{\A_\ast}}\bullet\mathfrak{e}\bullet(L\circ R))\\
&=\mathfrak{s}_2\circ(G_\ast^{\downarrow\y_{\A_\ast}}\bullet\mathfrak{e}\bullet(L\circ R)) &\textrm{[by (\ref{EquL})]}
\end{align*} 
Now that we have constructed the adjunction $L\dashv R$, whose unit is an isomorphism, we will make few more remarks needed to construct the desired $\B$-structure $H_G$.

\subsection*{\newstep\label{step cofiltered category} $(G\downarrow \y_\A^{\alpha,F})$ is a cofiltered category}

Since $\CC$ has finite limits and $\A_\ast$ is the underlying sketch of a finite category further equipped with some commutativity conditions, also $\A_\ast\CC$ and $\A_\ast\CCC$ have finite limits, computed component-wise. Moreover, the functor 
$$\y_{\A_\ast}\colon\A_\ast\CC\to\A_\ast\CCC$$
preserves them. This implies that $(G_\ast \downarrow \y_{\A_\ast })$ has finite limits. In view of the adjunction $L\dashv R$, with unit being an isomorphism, this implies
that $(G\downarrow \y_\A^{\alpha,F})$ also has finite limits. Since $(G\downarrow \y_\A^{\alpha,F})$ is in addition a small category, we get that it is a cofiltered category.

\subsection*{\newstep The functors $L_A$}

Now consider an object $A$ in $\A$ and the associated functor $L_A$ arising in the following diagram thanks to the universal property of the comma category on the right.
$$\vcenter{\cd@=30pt{ (G_\ast\downarrow \y_{\A_\ast})\ar@/_10pt/[ddd]\ar[dr]_-{G_\ast^{\downarrow \y_{\A_\ast}}}\ar@{-->}[rrr]^-{L_A} &\ar@{}[rd]|{=}  &  & (G(A)\downarrow \y)\ar[ld]^-{G(A)^{\downarrow \y}}\ar@/^10pt/[ddd] \\ 
 & \A_\ast\CC\ar[d]_-{\y_{\A_\ast}}\ar@{}[dr]|{=}\ar[r]^-{A_\CC} & \CC\ar[d]^-{\y} & \dltwocell<\omit>{<1>\quad\quad\quad\quad\quad(G(A),\y)} \\ 
 & {\A_\ast}\CCC \ar[r]_-{A_{\CCC}}\ultwocell<\omit>{<0>(G_\ast,\y_{\A_\ast})\;\quad\quad\quad}\ar@{}[rd]|{=}  & \CCC &  \\ \1\ar[ur]^-{G_\ast}\ar@{=}[rrr] & & & \1\ar[ul]_-{G(A)} }}$$
Since the middle square above satisfies the Beck-Chevalley condition relative to the right adjoints of the horizontal functors (see Subsection~\ref{subsection kan extensions}), we can conclude that $L_A$ has a right adjoint. 

\subsection*{\newstep\label{step limiting cone} $\mathsf{J}^\alpha_{\y F}\bullet (G,\y_{\A}^{\alpha,F})$ is a limiting cone}

For an object $A \in \A$, consider the commutative diagram
$$\cd{(G \downarrow \y_\A^{\alpha,F})\ar[rddd] \ar[r]^-{L_A\circ L} \ar[d]_-{G^{\downarrow\y_\A^{\alpha,F}}} & (G(A) \downarrow \y) \ar[d]^-{G(A)^{\downarrow\y}}\ar@{}[dl]^<<<<<<<<<<{=} \\ \A^\alpha_F\CC \ar[d]_-{\y_{\A}^{\alpha,F}} & \CC \ar[dd]^-{\y} \\
\A^\alpha_{\y F}\CCC \ar[d]_-{\mathsf{J}^\alpha_{\y F}} &  \\ \A\CCC  \ar[r]_-{A_{\CCC}}\ar@{}[ur]^<<<<<<<<<<{=} & \CCC.}$$
By Theorem~\ref{c tilde lemma 2}(\ref{cofiltered limit}), $(G(A),\y)$ is a limiting cone of the composite of functors in the right column of the above diagram. Since the top functor has a right adjoint, this implies that $$(G(A),\y)\bullet (L_A\circ L)=(A_{\CCC}\circ \mathsf{J}^\alpha_{\y F})\bullet (G,\y_{\A}^{\alpha,F})$$ is a limiting cone of the diagonal functor. Since $(G \downarrow \y_\A^{\alpha,F})$ is a cofiltered category, we can conclude that
\begin{equation}\label{EquS}
\mathsf{J}^\alpha_{\y F}\bullet (G,\y_{\A}^{\alpha,F})
\end{equation}
is a limiting cone of the composite of functors in the left column. The vertex of this limiting cone is the object $G\in\A\CCC$. We have now all the ingredients to prove that there exists a $\B$-structure $H_G$ in $\CCC$ such that $H_G\circ\beta$ is isomorphic to $G$ in $\A\CCC $, and moreover, $H_G$ is functorial and the isomorphism is natural in $G$.

\subsection*{\newstep The functor $\beta_F'$}

Since there is a functorial verification of $\alpha\vdash\beta$ for $F$, we know that $\beta^\alpha_F\colon \B^{\beta\alpha}_F\CC\to\A^\alpha_F\CC$ admits a right inverse. Let us denote it by $\beta_F'$.

\subsection*{\newstep\label{step definition HG} The $\B$-structure $H_G$ and the natural transformation $\mathfrak{p}_G$}

We then have the following commutative diagram, where at the bottom, $A$ represents an arbitrary object of $\A$:
$$\cd{ (G \downarrow \y^{\alpha,F}_{\A})\ar[d]^-{G^{\downarrow\y^{\alpha,F}_{\A}}}\ar@{=}[r]\ar@/_80pt/@{-->}[dddd]_-{C_G}\ar@{}[dddd]_-{=\quad\quad\quad\quad\quad} & (G \downarrow \y^{\alpha,F}_{\A})\ar@{}[ddl]|{=}\ar[dd]^-{G^{\downarrow\y^{\alpha,F}_{\A}}} \\ \A^\alpha_F\CC\ar[d]_-{\beta'_F}\ar[dr]|(.4){\hole}|-{1_{\A^\alpha_F\CC}}\ar@{}[ddr]|<<<<<<<<<<<<{=} &  \\
\B^{\beta\alpha}_F\CC\ar@{}[ddr]|{=} \ar[d]_-{\mathsf{J}^{\beta\alpha}_F} \ar@<-2pt>[r]_-{\beta^\alpha_F} & \A^\alpha_F\CC \ar[d]^-{\y^{\alpha,F}_{\A}} \\
\B\CC \ar[d]_-{\y_{\B}} & \A^\alpha_{\y F}\CCC \ar[d]^-{\mathsf{J}^\alpha_{\y F}} \\
\B\CCC \ar[r]^-{\beta_{\CCC}}\ar@{..>}[d]_-{\beta(A)_{\CCC}}\ar@{}[dr]|{=} & \A\CCC \ar@{..>}[d]^-{A_{\CCC}} \\ \CCC\ar@{=}[r] & \CCC }$$
Let $H_G$ be the cofiltered limit of $C_G$ in the diagram above, with $\mathfrak{p}_G$ denoting the limiting cone. We will show that the construction $H_G$ is functorial in $G$. 

\subsection*{\newstep\label{step the functor m square} The functor $m^\#$}

A morphism $m\colon G\to G'$ in $\A^\alpha_{\y F}\CCC$ gives rise to a functor $m^\#\colon (G' \downarrow \y^{\alpha,F}_{\A})\to(G \downarrow \y^{\alpha,F}_{\A})$, which is uniquely determined by the following pasting identity.
\begin{equation}\label{EquQ}\vcenter{\cd@=35pt{ & (G\downarrow \y_\A^{\alpha,F})\ar[d]^-{G^{\downarrow \y_\A^{\alpha,F}}}\ar[r]\drtwocell<\omit>{<1>\,\,\,\,\quad\quad(G,\y_\A^{\alpha,F})} & \1\ar[d]^-{G}\\ (G'\downarrow \y_{\A}^{\alpha,F})\ar@{}[urr]|<<<<<<<<<{=}\ar[r]_-{G'^{\downarrow \y_\A^{\alpha,F}}}\ar[ur]^-{m^\#} & \A^\alpha_F\CC\ar[r]_-{\y_{\A}^{\alpha,F}} & \A^\alpha_{\y F}\CCC }}\equiv\vcenter{\cd@=35pt{(G'\downarrow \y^{\alpha,F}_{\A})\ar[d]^-{G'^{\downarrow \y_\A^{\alpha,F}}}\ar[r]\drtwocell<\omit>{<1>\,\,\,\,\quad\quad\; (G',\y_\A^{\alpha,F})} & \1\ar[d]^(.25){G'}\ar@{=}[r] & \1\ar[dl]^-{G}\dltwocell<\omit>{<2> m\;}\\ \A^\alpha_F\CC\ar[r]_-{\y_{\A}^{\alpha,F}} & \A^\alpha_{\y F}\CCC & }}
\end{equation}

\subsection*{\newstep\label{step the morphism Hm} The morphism $H_m$}

This functor gives rise, in turn, to a morphism $H_m\colon H_G\to H_{G'}$ which is uniquely determined by the following pasting identity.
\begin{equation}\label{EquR}\vcenter{\cd@=30pt{ & \1\ar[dr]|(.25){\hole}|<<<<<{H_{G'}}|(.35){\hole}\ar@/^25pt/[dr]|(.4){\hole}|-{H_G}|(.53){\hole}\drtwocell<\omit>{<-2>\quad H_m} & \\ (G'\downarrow \y_\A^{\alpha,F}) \ar[rr]_-{C_{G'}}\ar[ur]\urrtwocell<\omit>{<1>\quad\mathfrak{p}_{G'}} &  & \B\CCC}}\quad\quad\equiv\quad\vcenter{ \cd{ & & \1\ar[dr]^-{H_G} & \\ (G'\downarrow \y_\A^{\alpha,F})\ar[r]_-{m^\#} & (G\downarrow \y_\A^{\alpha,F}) \ar[rr]_-{C_G}\ar[ur]\urrtwocell<\omit>{<1>\quad\mathfrak{p}_{G}} &  & \B\CCC}}
\end{equation}
This establishes functoriality of $H_G$ in $G$.

\subsection*{\newstep\label{step beta p limiting cone} $\beta_{\CCC}\bullet\mathfrak{p}_G$ is a limiting cone}

It remains to construct an isomorphism $i^G\colon H_G\circ\beta\cong G$ natural in $G$. For this, we first claim that for each $G$, the cone $\beta_{\CCC}\bullet\mathfrak{p}_G$ is a limiting cone. Since $(G \downarrow \y_\A^{\alpha,F})$ is a cofiltered category, and the evaluation functors $A_{\CCC}$ jointly reflect cofiltered limits, it is sufficient to show that each cone $$A_{\CCC}\bullet(\beta_{\CCC}\bullet\mathfrak{p}_G)=(A_{\CCC}\circ\beta_{\CCC})\bullet\mathfrak{p}_G=\beta(A)_{\CCC}\bullet \mathfrak{p}_G$$ is a limiting cone. This is indeed the case since the evaluation functors $\beta(A)_{\CCC}$ preserve cofiltered limits.

\subsection*{\newstep\label{step the isomorphism iG} The isomorphism $i^G$}

Now the limiting cone $\beta_{\CCC}\bullet\mathfrak{p}_G$ is over the same functor as the limiting cone~(\ref{EquS}). So there must be an isomorphism $i^G\colon H_G\circ\beta=\beta_{\CCC}(H_G)\cong G$ uniquely determined by the identity
\begin{equation}\label{EquT}
(\mathsf{J}^\alpha_{\y F}\bullet (G,\y_{\A}^{\alpha,F}))\circ(i^G\bullet !_{(G\downarrow \y_\A^{\alpha,F})})=\beta_{\CCC}\bullet\mathfrak{p}_G.
\end{equation}

\subsection*{\newstep\label{step the naturality of iG}\label{last step} The naturality of $i^G$}

For the naturality of this isomorphism, we must show that the diagram 
$$\cd{ H_G\circ\beta\ar[r]^-{i^G}\ar[d]_-{H_m\bullet\beta} & G\ar[d]^-{m} \\ H_{G'}\circ\beta\ar[r]_-{i^{G'}} & G' }$$
commutes for each morphism $m\colon G\to G'$ in $\A^\alpha_{\y F}\CCC$. For this it suffices to show that 
$$(\mathsf{J}\bullet (G',\y_{\A}^{\alpha,F}))\circ((i^{G'}\circ(H_m\bullet\beta))\bullet !^{G'})=(\mathsf{J}\bullet (G',\y_{\A}^{\alpha,F}))\circ((m\circ i^G)\bullet !^{G'}),$$
where $\mathsf{J}=\mathsf{J}^\alpha_{\y F}$ and $!^{G'}=!_{(G'\downarrow \y_\A^{\alpha,F})}$ (similarly, in what follows, $!^{G}=!_{(G\downarrow \y_\A^{\alpha,F})}$). Indeed, we have:
{\allowdisplaybreaks
\begin{align*}
(\mathsf{J}\bullet (G',\y_{\A}^{\alpha,F}))\circ((i^{G'}\circ(H_m\bullet\beta))\bullet !^{G'}) & = (\mathsf{J}\bullet (G',\y_{\A}^{\alpha,F}))\circ(i^{G'}\bullet !^{G'})\circ((H_m\bullet\beta)\bullet !^{G'}) \\*
&= (\beta_{\CCC}\bullet\mathfrak{p}_{G'})\circ((H_m\bullet\beta)\bullet !^{G'}) \textrm{\hspace{27pt} [by (\ref{EquT}) for $G'$]}\\
&= (\beta_{\CCC}\bullet\mathfrak{p}_{G'})\circ(\beta_{\CCC}\bullet H_m\bullet !^{G'})\\
&= \beta_{\CCC}\bullet(\mathfrak{p}_{G'}\circ(H_m\bullet !^{G'}))\\
&= \beta_{\CCC}\bullet\mathfrak{p}_{G}\bullet m^\# \textrm{\hspace{123pt} [by (\ref{EquR})]}\\
&=((\mathsf{J}\bullet (G,\y_{\A}^{\alpha,F}))\circ(i^{G}\bullet !^{G}))\bullet m^\# \textrm{\hspace{37pt} [by (\ref{EquT})]} \\
&=(\mathsf{J}\bullet (G,\y_{\A}^{\alpha,F})\bullet m^\#)\circ(i^{G}\bullet !^{G'}) \\
&=(\mathsf{J}\bullet ((G',\y_{\A}^{\alpha,F})\circ (m\bullet!^{G'})))\circ(i^{G}\bullet !^{G'}) \textrm{\hspace{13pt} [by (\ref{EquQ})]} \\
&=(\mathsf{J}\bullet (G',\y_{\A}^{\alpha,F}))\circ (\mathsf{J}\bullet m\bullet!^{G'})\circ(i^{G}\bullet !^{G'})\\
&=(\mathsf{J}\bullet (G',\y_{\A}^{\alpha,F}))\circ (((\mathsf{J}\bullet m)\circ i^{G})\bullet !^{G'})\\*
&=(\mathsf{J}\bullet (G',\y_{\A}^{\alpha,F}))\circ ((m\circ i^{G})\bullet !^{G'})
\end{align*}}%
With this the proof of the stability theorem is complete.

\section{Coherence}\label{section coherence}

In this section we add a technical remark to Theorem~\ref{C tilde preserves unconditional exactness properties} that, in fact, under the conditions of the theorem, the right inverse of $\beta_{\y F}^{\alpha}$ can be chosen so that it agrees with the given right inverse of $\beta_F^{\alpha}$ in the following sense. Using the notation from the proof above, let us consider the factorisation $\y^{\beta\alpha,F}_{\B}$ making the diagram
\begin{equation}\label{EquX}\vcenter{\cd@=30pt{\B^{\beta\alpha}_F\CC \ar[dd]_-{\beta^\alpha_F} \ar@{-->}[rrr]^-{\y^{\beta\alpha,F}_{\B}} \ar[rd]|-{\mathsf{J}^{\beta\alpha}_F} \ar@{}[ddr]|-{\mathrm{pb}} &&& \B^{\beta\alpha}_{\y F}\CCC \ar[dd]^-{\beta^\alpha_{\y F}} \ar[ld]|-{\mathsf{J}^{\beta\alpha}_{\y F}} \ar@{}[ddl]|-{\mathrm{pb}} \\ & \B\CC \ar[r]^-{\y_\B} \ar[d]_(.4){\beta_\CC} \ar@{}[dr]|{=} \ar@{}[ru]|{=} & \B\CCC \ar[d]^(.4){\beta_{\CCC}} & \\ \A^\alpha_F\CC \ar[r]^-{\mathsf{J}^\alpha_F} \ar[dd] \ar@{}[dr]|-{\mathrm{pb}} & \A\CC \ar[r]^-{\y_\A} \ar[d]_-{\alpha_\CC} \ar@{}[dr]|{=} & \A\CCC \ar[d]^-{\alpha_{\CCC}} & \A^\alpha_{\y F}\CCC \ar[l]_(.55){\mathsf{J}^\alpha_{\y F}} \ar[dd] \ar@{}[dl]|-{\mathrm{pb}} \\ & \X\CC \ar[r]_-{\y_\X} \ar@{}[rd]|{=} & \X\CCC & \\ \1 \ar[ru]^-{F} \ar@{=}[rrr] &&& \1 \ar[lu]_-{\y F}}}\end{equation}
commutative, where the $\mathsf{J}$ functors are subcategory inclusions. Let us notice that the rectangle
$$\cd{\B^{\beta\alpha}_F\CC \ar[d]_-{\beta^\alpha_F} \ar[r]^-{\y^{\beta\alpha,F}_{\B}} & \B^{\beta\alpha}_{\y F}\CCC \ar[d]^-{\beta^\alpha_{\y F}} \\ \A^\alpha_F\CC \ar[r]_-{\y_\A^{\alpha,F}} & \A^\alpha_{\y F}\CCC}$$
is also commutative. Given a right inverse $\beta'_F$ of $\beta^\alpha_F$, we showed in the above proof the existence of a right inverse $\beta'_{\y F}$ of $\beta^\alpha_{\y F}$. Let us now prove that these functorial verifications agree (are coherent) with each other; i.e., with an appropriate choice of limits, they make the rectangle
$$\cd{\B^{\beta\alpha}_F\CC \ar[r]^-{\y^{\beta\alpha,F}_{\B}} & \B^{\beta\alpha}_{\y F}\CCC \\ \A^\alpha_F\CC \ar[u]^-{\beta'_F} \ar[r]_-{\y_\A^{\alpha,F}} & \A^\alpha_{\y F}\CCC \ar[u]_-{\beta'_{\y F}}}$$
commutative. For each $K \in \A^\alpha_F\CC$, the pasting
\begin{equation*}\vcenter{\cd@=34pt{(K\downarrow \A^\alpha_F\CC) \ar[r] \ar[d]_-{K^{\downarrow\A^\alpha_F\CC}} \drtwocell<\omit>{<1>\qquad\,\,\,\,\,\,(K,\A^\alpha_F\CC)} & \1 \ar[d]^-{K} \\ \A^\alpha_F\CC \ar[r]|-{1_{\A^\alpha_F\CC}} \ar[d]_-{1_{\A^\alpha_F\CC}} \ar@{}[rd]|-{\mathrm{pb}} & \A^\alpha_F\CC \ar[d]^-{\y_\A^{\alpha,F}} \\ \A^\alpha_F\CC \ar[r]_-{\y_\A^{\alpha,F}} & \A^\alpha_{\y F}\CCC}}\end{equation*}
where $(K\downarrow \A^\alpha_F\CC)$, $K^{\downarrow\A^\alpha_F\CC}$ and $(K,\A^\alpha_F\CC)$ denote $(K\downarrow 1_{\A^\alpha_F\CC})$, $K^{\downarrow 1_{\A^\alpha_F\CC}}$ and $(K,1_{\A^\alpha_F\CC})$ respectively, is isomorphic to the comma category $(\y_\A^{\alpha,F}(K)\downarrow \y_\A^{\alpha,F})$ since they satisfy the same universal property.
Let us denote this isomorphism by $\y_K^{\textrm{iso}} \colon (K\downarrow \A^\alpha_F\CC) \to (\y_\A^{\alpha,F}(K)\downarrow \y_\A^{\alpha,F})$. It is uniquely determined by the following pasting identity.
\begin{equation}\label{EquW}\vcenter{\cd@=50pt{(K \downarrow \A^\alpha_F\CC) \ar[d]_-{\y_K^{\textrm{iso}}} \ar@<4pt>@/^1pc/[rd]^-{K^{\downarrow \A^\alpha_F\CC}} \ar@{}[rdd]|(.3){=} & \\ (\y_\A^{\alpha,F}(K)\downarrow \y_\A^{\alpha,F}) \ar[r]^-{\y_\A^{\alpha,F}(K)^{\downarrow \y_\A^{\alpha,F}}}="B" \ar[d] & \A^\alpha_F\CC \ar[d]^-{\y_\A^{\alpha,F}} \\ \1 \ar[r]_-{\y_\A^{\alpha,F}(K)}="A" & \A^\alpha_{\y F}\CCC \ultwocell<\omit> \ar@{}"A";"B"|(.37){(\y_\A^{\alpha,F}(K),\y_\A^{\alpha,F})}}} \quad \equiv \quad \vcenter{\cd@=50pt{(K \downarrow \A^\alpha_F\CC) \ar[r]^-{K^{\downarrow \A^\alpha_F\CC}} \ar[d] & \A^\alpha_F\CC \ar[d]^-{1_{\A^\alpha_F\CC}} \\ \1 \ar[r]_-{K} & \A^\alpha_F\CC \ar[d]^-{\y_\A^{\alpha,F}} \ultwocell<\omit>{<-1>\qquad\qquad\quad\,\,\,\,\,(K,\A^{\alpha}_F\CC)} \\ & \A^\alpha_{\y F}\CCC}}\end{equation}
The universal property of the comma category gives rise to a functor $I \colon \1 \to (K\downarrow \A^\alpha_F\CC)$ uniquely determined by the following pasting identity.
\begin{equation}\label{EquV}\vcenter{\cd@=40pt{ & (K\downarrow \A^\alpha_F\CC)\ar[d]|-{K^{\downarrow\A^\alpha_F\CC}}\ar[r]\drtwocell<\omit>{<1>\qquad\quad(K,\A^\alpha_F\CC)} & \1\ar[d]^-{K}\\ \1\ar@{}[urr]|(.25){=}\ar[r]_-{K}\ar[ur]^-{I} & \A^\alpha_F\CC\ar[r]_-{1_{\A^\alpha_F\CC}} & \A^\alpha_F\CC}}\,\,\equiv\,\,\vcenter{\cd@=40pt{\1\ar[d]_-{K}\ar[r]\ar@{}[rd]|{=} & \1\ar[d]^-{K} \\ \A^\alpha_F\CC\ar[r]_-{1_{\A^\alpha_F\CC}} & \A^\alpha_F\CC & }}\end{equation}
Let us now prove that the cone
$$\cd@=20pt{&& \1 \ar[rrdd]^-{K} \\ \\ (K\downarrow \A^\alpha_F\CC) \ar[rrrr]_-{K^{\downarrow\A^\alpha_F\CC}} \ar[rruu] & \rtwocell<\omit>{<-5>\qquad\,\,\,\,\,\,(K,\A^\alpha_F\CC)} &&& \A^\alpha_F\CC \ar[rrrr]_-{\y^\B \circ \mathsf{J}^{\beta\alpha}_F \circ \beta'_F} &&&& \B\CCC}$$
over $C_{\y_\A^{\alpha,F}(K)} \circ \y_K^{\textrm{iso}}$ is a limiting cone.
Since $I$ is an initial object in $(K\downarrow \A^\alpha_F\CC)$, this is the case if and only if the cone
$$\cd@=20pt{&&&& \1 \ar[rrdd]^-{K} \\ \\\1 \ar[rr]_-{I} && (K\downarrow \A^\alpha_F\CC) \ar[rrrr]_-{K^{\downarrow\A^\alpha_F\CC}} \ar[rruu] & \rtwocell<\omit>{<-5>\qquad\,\,\,\,\,\,(K,\A^\alpha_F\CC)} &&& \A^\alpha_F\CC \ar[rrrr]_-{\y^\B \circ \mathsf{J}^{\beta\alpha}_F \circ \beta'_F} &&&& \B\CCC}$$
is a limiting cone. In view of~(\ref{EquV}), this cone is just the identity over a functor whose domain is~$\1$, which implies it is a limiting cone.
Therefore, since $\y_\A^{\alpha,F}$ is injective on objects, we can choose, for each $K \in \A^\alpha_F\CC$, $H_{\y_\A^{\alpha,F}(K)}$ to be
$$H_{\y_\A^{\alpha,F}(K)}=(\y_\B \circ \mathsf{J}^{\beta\alpha}_F \circ \beta'_F)(K)=(\mathsf{J}^{\beta\alpha}_{\y F} \circ \y^{\beta\alpha,F}_{\B} \circ \beta'_F)(K)$$
with \begin{equation}\label{EquZ}\mathfrak{p}_{\y_\A^{\alpha,F}(K)} = (\y_\B \circ \mathsf{J}^{\beta\alpha}_F \circ \beta'_F) \bullet (K,\A^\alpha_F\CC) \bullet (\y_K^{\textrm{iso}})^{-1}.\end{equation}
Now, with these choices, we would like to show that $i^{\y_\A^{\alpha,F}(K)}=1_{\y_\A^{\alpha,F}(K)}$. In view of the definition of $i^{\y_\A^{\alpha,F}(K)}$ from~(\ref{EquT}), this follows from the equalities:
\begin{align*}
&\mathsf{J}^\alpha_{\y F} \bullet (\y_\A^{\alpha,F}(K),\y_\A^{\alpha,F})\\
&= \mathsf{J}^\alpha_{\y F} \bullet (\y_\A^{\alpha,F} \bullet (K,\A^\alpha_F\CC) \bullet (\y_K^{\textrm{iso}})^{-1}) &\textrm{[by (\ref{EquW})]}\\
&= (\mathsf{J}^\alpha_{\y F} \circ \y_\A^{\alpha,F}) \bullet (K,\A^\alpha_F\CC) \bullet (\y_K^{\textrm{iso}})^{-1}\\
&= (\y_\A \circ \mathsf{J}^\alpha_F) \bullet (K,\A^\alpha_F\CC) \bullet (\y_K^{\textrm{iso}})^{-1} & \textrm{[by (\ref{EquD})]}\\
&= (\y_\A \circ \mathsf{J}^\alpha_F \circ \beta^\alpha_F \circ \beta'_F) \bullet (K,\A^\alpha_F\CC) \bullet (\y_K^{\textrm{iso}})^{-1}\\
&= (\beta_{\CCC} \circ \y_\B \circ \mathsf{J}^{\beta\alpha}_F \circ \beta'_F) \bullet (K,\A^\alpha_F\CC) \bullet (\y_K^{\textrm{iso}})^{-1} &\textrm{[by (\ref{EquX})]}\\
&= \beta_{\CCC} \bullet \mathfrak{p}_{\y_\A^{\alpha,F}(K)} &\textrm{[by (\ref{EquZ})]}
\end{align*}
In view of the proof of Lemma~\ref{LemB}, $i^{\y_\A^{\alpha,F}(K)}=1_{\y_\A^{\alpha,F}(K)}$ implies that $$\beta'_{\y F}(\y_\A^{\alpha,F}(K))=H_{\y_\A^{\alpha,F}(K)}=(\y^{\beta\alpha,F}_\B \circ \beta'_F)(K).$$
We still need to prove the identity $\beta'_{\y F} \circ \y_\A^{\alpha,F} = \y^{\beta\alpha,F}_\B \circ \beta'_F$ holds on morphisms. So let $n \colon K \to K'$ be any morphism in $\A^\alpha_F\CC$. We are required to show
$$H_{\y_\A^{\alpha,F}(n)}=(\mathsf{J}^{\beta\alpha}_{\y F} \circ \y^{\beta\alpha,F}_\B \circ \beta'_F)(n).$$
Since $\y_\A^{\alpha,F}$ is faithful and the equalities
\begin{align*}
&\y_\A^{\alpha,F} \bullet (K,\A^\alpha_F\CC) \bullet ((\y_K^{\textrm{iso}})^{-1} \circ (\y_\A^{\alpha,F}(n))^\#)\\
&= (\y_\A^{\alpha,F} \bullet (K,\A^\alpha_F\CC) \bullet (\y_K^{\textrm{iso}})^{-1}) \bullet (\y_\A^{\alpha,F}(n))^\#\\
&= (\y_\A^{\alpha,F}(K),\y_\A^{\alpha,F}) \bullet (\y_\A^{\alpha,F}(n))^\# &\textrm{[by (\ref{EquW})]}\\
&= (\y_\A^{\alpha,F}(K'),\y_\A^{\alpha,F}) \circ (\y_\A^{\alpha,F}(n) \bullet !_{(\y_\A^{\alpha,F}(K') \downarrow \y_\A^{\alpha,F})}) &\textrm{[by (\ref{EquQ})]}\\
&= (\y_\A^{\alpha,F} \bullet (K',\A^\alpha_F\CC) \bullet (\y_{K'}^{\textrm{iso}})^{-1}) \circ (\y_\A^{\alpha,F}(n) \bullet !_{(\y_\A^{\alpha,F}(K') \downarrow \y_\A^{\alpha,F})}) &\textrm{[by (\ref{EquW}) for $K'$]}\\
&= (\y_\A^{\alpha,F} \bullet (K',\A^\alpha_F\CC) \bullet (\y_{K'}^{\textrm{iso}})^{-1}) \circ (\y_\A^{\alpha,F} \bullet n \bullet !_{(\y_\A^{\alpha,F}(K') \downarrow \y_\A^{\alpha,F})})\\
&= \y_\A^{\alpha,F} \bullet (((K',\A^\alpha_F\CC) \bullet (\y_{K'}^{\textrm{iso}})^{-1}) \circ (n \bullet !_{(\y_\A^{\alpha,F}(K') \downarrow \y_\A^{\alpha,F})}))
\end{align*}
hold, the identity
\begin{equation}\label{EquY}
(K,\A^\alpha_F\CC) \bullet ((\y_K^{\textrm{iso}})^{-1} \circ (\y_\A^{\alpha,F}(n))^\#) = ((K',\A^\alpha_F\CC) \bullet (\y_{K'}^{\textrm{iso}})^{-1}) \circ (n \bullet !_{(\y_\A^{\alpha,F}(K') \downarrow \y_\A^{\alpha,F})})
\end{equation}
holds. In view of
{\allowdisplaybreaks
\begin{align*}
&\mathfrak{p}_{\y_\A^{\alpha,F}(K')} \circ ((\mathsf{J}^{\beta\alpha}_{\y F} \circ \y^{\beta\alpha,F}_\B \circ \beta'_F)(n) \bullet !_{(\y_\A^{\alpha,F}(K')\downarrow \y_\A^{\alpha,F})})\\*
&= \mathfrak{p}_{\y_\A^{\alpha,F}(K')} \circ ((\mathsf{J}^{\beta\alpha}_{\y F} \circ \y^{\beta\alpha,F}_\B \circ \beta'_F) \bullet n \bullet !_{(\y_\A^{\alpha,F}(K')\downarrow \y_\A^{\alpha,F})})\\
&= \mathfrak{p}_{\y_\A^{\alpha,F}(K')} \circ ((\y_\B \circ \mathsf{J}^{\beta\alpha}_F \circ \beta'_F) \bullet n \bullet !_{(\y_\A^{\alpha,F}(K')\downarrow \y_\A^{\alpha,F})}) &\textrm{[by (\ref{EquX})]}\\
&= ((\y_\B \circ \mathsf{J}^{\beta\alpha}_F \circ \beta'_F) \bullet (K',\A^\alpha_F\CC) \bullet (\y_{K'}^{\textrm{iso}})^{-1})\\*
&\qquad \circ ((\y_\B \circ \mathsf{J}^{\beta\alpha}_F \circ \beta'_F) \bullet n \bullet !_{(\y_\A^{\alpha,F}(K')\downarrow \y_\A^{\alpha,F})}) &\textrm{[by (\ref{EquZ}) for $K'$]}\\
&= (\y_\B \circ \mathsf{J}^{\beta\alpha}_F \circ \beta'_F) \bullet (((K',\A^\alpha_F\CC) \bullet (\y_{K'}^{\textrm{iso}})^{-1}) \circ (n \bullet !_{(\y_\A^{\alpha,F}(K')\downarrow \y_\A^{\alpha,F})}))\\
&= (\y_\B \circ \mathsf{J}^{\beta\alpha}_F \circ \beta'_F) \bullet (K,\A^\alpha_F\CC) \bullet ((\y_K^{\textrm{iso}})^{-1} \circ (\y_\A^{\alpha,F}(n))^\#) &\textrm{[by (\ref{EquY})]}\\*
&= \mathfrak{p}_{\y_\A^{\alpha,F}(K)} \bullet (\y_\A^{\alpha,F}(n))^\# &\textrm{[by (\ref{EquZ})]}
\end{align*}}%
and the definition of $H_{\y_\A^{\alpha,F}(n)}$ from~(\ref{EquR}), we have $H_{\y_\A^{\alpha,F}(n)}=(\mathsf{J}^{\beta\alpha}_{\y F} \circ \y^{\beta\alpha,F}_\B \circ \beta'_F)(n)$ as required. This shows that $\beta'_{\y F} \circ \y_\A^{\alpha,F} = \y^{\beta\alpha,F}_\B \circ \beta'_F$.

\section{Concluding remarks}\label{section conclusion}

In this section we discuss a few topics for possible future research in the subject of exactness properties.

\subsection{Exactness varieties of categories}

The study of exactness properties share similarities with universal algebra. There, for a fixed signature, one relates classes of identities expressed using operators from the signature and classes of algebras satisfying identities. Similarly here, it would be interesting to investigate the relationship between classes of exactness properties and classes of categories having exactness properties. In particular, which classes of (say, finitely complete) categories are `exactness varieties' of categories in the sense that they include all categories having all exactness properties in a fixed class of exactness properties, and only those? In other words, is there an analogue of Birkhoff Variety Theorem for exactness properties?

\subsection{Essentially algebraic categories}

Our stability theorem reinforces the idea, which emerged already in \cite{jacqminthesis,jacqmin1,jacqmin2}, that the study of exactness properties in the context of finitely complete categories is intimately linked with the theory of essentially algebraic categories in the sense of~\cite{AHR,AR}. Indeed, the dual categories of pro-completions of finitely complete small categories are nothing but essentially algebraic categories. Therefore, by iteration of pro-completion with its dual (up to omitted changes of universes for size issues),
$$\CC\hookrightarrow \LC\hookrightarrow \Lex(\Lex(\CC,\Set),\Set)$$ we obtain a representation of any finitely complete category $\CC$ as a full subcategory of an essentially algebraic category. Moreover, the embedding functor preserves finite limits, finite colimits and reflects isomorphisms. So by our stability theorem (applied twice), we get that, in many cases, working in a finitely complete category having a certain exactness property can be reduced to working in an essentially algebraic category having the same property. The role that essentially algebraic categories play for the study of exactness properties should be better understood.

\subsection{Higher-order exactness properties}

In the Introduction, we referred to exactness properties expressed by exactness sequents defined in this paper as `first-order' exactness properties. This terminology is suggested by the fact that in topos theory, the properties of a topos which are usually referred to as `first-order' properties are roughly those which can be expressed by our exactness sequents. Definition of higher-order exactness properties and the question of their stability under pro-completion are to be looked at. There are a number of examples of higher-order exactness properties not only in classical topos theory, but also in recent developments in categorical algebra (e.g.\ those introduced in \cite{BJK,BJ,Gray}).

\subsection{Exactness properties for Mal'tsev conditions}

In universal algebra, a `Mal'tsev condition' on a variety of algebras is a condition which is equivalent to the existence of certain terms in its algebraic theory satisfying certain term identities. Some Mal'tsev conditions can be expressed as first-order exactness properties of the corresponding category of algebras (see e.g.~\cite{janelidze5}). The question here is: given a Mal'tsev condition on a variety of algebras, can one find a first-order exactness property of a category which is equivalent to the given Mal'tsev condition in the restricted case of varieties of algebras? Experience working with Mal'tsev conditions that are known to be expressible as first-order exactness properties suggests that if a Mal'tsev condition involves equations where the same terms occur at different bracket depth, then such a Mal'tsev condition cannot be expressed as a first-order exactness property. In particular, the following is an open problem: is the Mal'tsev condition of existence of terms of the algebraic theory of groups expressible as a first-order exactness property? Note that the associativity axiom of the group multiplication has the multiplication term occurring at two different bracket levels. A related line of research is to identify exactness properties of an object which describe certain algebraically defined members of a variety. A striking positive result in this area is given in~\cite{MRVdL,GM}, where it is shown that groups can be identified in the variety of monoids using a first-order exactness property.

\subsection{Enriched and structured categories}

It would be interesting to explore extensions of our stability theorem (and more generally, the theory of exactness properties) when a category is replaced with another categorical structure, such as an enriched category or a monoidal category, see e.g.~\cite{DS,DS2}. In fact, the earlier version of our stability theorem was formulated for algebraically enriched categories (see~\cite{jacqminthesis}). Since algebraic enrichments can themselves be seen as exactness properties, as we remark at the end of Subsection~\ref{section examples}, the present formulation of the theorem only slightly drops generality of the context; aside from the context, the present theorem is actually significantly more general (for instance, its earlier version concerned exactness properties of empty structures). Another direction for generalization is considering exactness properties relative to special classes of morphisms, spans, etc. Sometimes this allows to unify various features of different exactness properties and generate new examples, see e.g. \cite{BMFMS,granjanurs,grandis,jacqmin3,NMF}. In general, a systematic investigation of exactness properties for `structured categories' should be an interesting direction for research.

\subsection{Completions}

There are many other types of completions arising in category theory, see e.g. \cite{ABLR,ARV,borceux,CC,CV,GL,hu,lack,lawvere,SC}. An interesting question would be to study which exactness properties are stable under these different completions. Particular instances of such stability results have been already established in~\cite{carboni,GJ,GranL,jacqminthesis,jacqmin2}. Even for pro-completion the question is not yet complete. Indeed, on one hand, one may wonder if the assumption on finite limits can be removed from Theorem~\ref{C tilde preserves unconditional exactness properties} (and thus replacing $\Lex(\CC,\Set)^{\op}$ by the pro-completion of~$\CC$). This question has been raised to us by Johnstone and remains open. On the other hand, as mentioned at the end of the Introduction, expressing an exactness property in a suitable way to apply our stability theorem may not be an easy task. This may be the reason why stability under the pro-completion for some particular exactness properties is still unknown, e.g.\ for being a Mal'tsev category, being regular protomodular, being protomodular, being strongly protomodular~\cite{bourn3}, and being a Gumm category~\cite{BG2}. Or perhaps, there are counterexamples to show that these properties are not stable under the pro-completion.



\vspace{30pt}
\begin{tabular}{rl}
Email: & pierre-alain.jacqmin@uclouvain.be\\
& zurab@sun.ac.za
\end{tabular}

\end{document}